\documentclass[a4paper,11pt]{amsart}
\usepackage{amsfonts,amssymb,amsmath,amsthm,abstract,color,enumerate}
\usepackage[ps,all,arc,rotate]{xy}
\usepackage[lmargin=1in,rmargin=1in,tmargin=1in,bmargin=1in]{geometry}
\usepackage{fancyhdr}
\usepackage{pb-diagram}

\newtheorem{prop}{Proposition}[section]
\newtheorem{lemma}[prop]{Lemma}
\newtheorem{thm}[prop]{Theorem}
\newtheorem{cor}[prop]{Corollary}

\theoremstyle{definition}
\newtheorem{defn}[prop]{Definition}

\newtheorem{rmk}[prop]{Remark}
\newtheorem{ex}[prop]{Example}

\DeclareMathOperator{\pic}{Pic}
\DeclareMathOperator{\Cox}{Cox}
\DeclareMathOperator{\spec}{Spec}

\DeclareMathOperator{\sym}{Sym}
\DeclareMathOperator{\Lie}{Lie}
\DeclareMathOperator{\sing}{Sing}
\newcommand{\ra}{\rightarrow}

\def\cO{\mathcal O}

\def\AA{\mathbb A}\def\CC{\mathbb C}
\def\GG{\mathbb G}\def\HH{\mathbb H}

\def\PP{\mathbb P}
\def\RR{\mathbb R}

\def\fg{\mathfrak g}
\def\fk{\mathfrak k}\def\fl{\mathfrak l}

\def\fs{\mathfrak s}

\def\SU{\mathrm{SU}} \def\GL{\mathrm{GL}} \def\SL{\mathrm{SL}}

\title{Algebraic symplectic analogues of additive quotients}

\author{Brent Doran and Victoria Hoskins}

\begin{document}
\maketitle

\begin{abstract}
Motivated by the study of hyperk\"{a}hler structures in moduli problems and hyperk\"{a}hler implosion, we initiate the study of non-reductive hyperk\"{a}hler and algebraic symplectic quotients with an eye towards those naturally tied to projective geometry, like cotangent bundles of blow-ups of linear arrangements of projective space. In the absence of a Kempf-Ness theorem for non-reductive quotients, we focus on constructing algebraic symplectic analogues of additive quotients of affine spaces, and obtain hyperk\"{a}hler structures on large open subsets of these analogues by comparison with reductive analogues. We show that the additive analogue naturally arises as the central fibre in a one-parameter family of isotrivial but non-symplectomorphic varieties coming from the variation of the level set of the moment map.  Interesting phenomena only possible in the non-reductive theory, like non-finite generation of rings, already arise in easy examples, but do not substantially complicate the geometry.
\end{abstract}

\section{Introduction}

Some of the most fascinating moduli spaces arise as hyperk\"{a}hler or algebraic symplectic analogues of simpler moduli spaces. The classical example is the moduli space of Higgs bundles over a curve, which contains the cotangent bundle to the moduli space of vector bundles as a dense open subset. The theory of Higgs bundles and the Hitchin fibration is extraordinarily fruitful in several areas of mathematics  \cite{hausel, hitchin}. Another example is Nakajima quiver varieties \cite{nakajima}, which are algebraic symplectic analogues of moduli spaces of quiver representations, and play an important role in representation theory. Each of these spaces, along with their analogues, arises via a quotient procedure from a group action on an affine space.  These examples suggests that hyperk\"{a}hler analogues of other moduli problems, particularly those arising as quotients of affine spaces, might bear similar fruit.  One could dream that such a picture may emerge in the other archetypical moduli problem, the moduli of curves.  Independently, especially given work on hyperk\"{a}hler analogues of toric varieties, it is worth asking how much of the basic structure of hyperk\"{a}hler quotients can be extended to nice enough projective varieties and their cotangent bundles.  In fact, it has recently been shown that a wide range of projective varieties, including blow-ups of projective space along projective linear arrangements, can be described as a quotient of affine spaces by a solvable group action \cite{brentnoah}.  The moduli space of stable pointed rational curves $\overline{M}_{0,n}$ admits such a description; so one can begin to explore the dream.  

On the one hand, hyperk\"{a}hler quotients by torus actions have been extensively studied for many years.  On the other, hyperk\"{a}hler implosion, recently developed in a series of papers by Dancer, Kirwan and their co-authors, can be thought of as addressing very special cases of unipotent actions via the study of a universal imploded cross-section \cite{DKR,DKS,DKS2}.  Our motivation is to broaden the class of unipotent actions amenable to a hyperk\"{a}hler quotient-like construction, while staying tied to smooth projective geometry, by virtue of using the full solvable group action.  In effect, the unipotent actions to consider are graded in the sense of \cite{BK}, where the torus action arises via the universal torsor of the projective variety as in \cite{brentnoah}.  One can then elide the technical difficulties of the stratified singularities that arise in implosion, because these are contained in unstable loci for the torus action, for a suitable ample character.  The present paper initiates this general study by focusing on the base case of linear actions of the additive group which, when suitably modified to non-linear additive actions and put in a family together with a N\'{e}ron-Severi torus action, suffice to realise $\overline{M}_{0,n}$ as a quotient of affine space.  

The theory of algebraic symplectic and hyperk\"{a}hler analogues for reductive quotients of affine spaces is reasonably well understood following work of Proudfoot \cite{proudfootphd}, which unifies the construction of hypertoric varieties and Nakajima quiver varieties. The construction of these analogues is given by an algebraic symplectic or hyperk\"{a}hler reduction of the cotangent lift of the action, where one takes a quotient of a level set of a moment map on the cotangent bundle. Over the complex numbers, the algebraic symplectic and hyperk\"{a}hler analogues of a reductive quotient of an affine space are homeomorphic by the Kempf--Ness Theorem \cite{kempf_ness}, which relies on the fact that a complex reductive group is the complexification of its maximal compact subgroup. More precisely, the algebraic symplectic analogue is constructed by taking a geometric invariant theory (GIT) quotient of the zero level set of an algebraic moment map by the reductive group and the hyperk\"{a}hler analogue is a topological quotient by the maximal compact group of the intersection of the algebraic moment map zero level set with a smooth moment map level set.

For non-reductive groups, one does not have a Kempf--Ness Theorem relating GIT quotients and symplectic reductions and, although one may expect such a result, it is not clear what shape it should take, as the maximal compact subgroup of the additive group is trivial. Until this hope is realised, we can only focus on the algebraic symplectic side; however, we will see below that the additive algebraic symplectic analogue inherits a hyperk\"{a}hler structure via a comparison with a reductive algebraic symplectic analogue.

In the simplest cases the additive theory provides new phenomena. We show how to produce examples where the ring of invariant functions on the zero level set of the algebraic moment map is non-finitely generated, and yet still work with the spaces in practice without difficulty.  Strikingly, as we show in Theorems \ref{thm fg nonzero} and \ref{1-param family red}, the additive theory, in contrast to the reductive theory, gives a natural one-parameter family $\mathcal{X} \ra \CC^*$ constructed as smooth algebraic symplectic reductions at non-zero values in $\CC^*$, which is isotrivial as a family of varieties, but not as a family of algebraic symplectic varieties. The central fibre of this family is singular and is the extrinsic quotient of the zero level set obtained by using invariant functions on the cotangent bundle (although, there exist invariant functions on the zero level set which do not extend to global invariant functions on the cotangent bundle). The algebraic symplectic analogue should be a quotient of the zero level set of the algebraic moment map, which is singular. However, we avoid these singularities by presenting an algebraic symplectic reduction which is a quotient of a dense open subset of the zero level set, determined by stability analysis; in Proposition  \ref{nonred alg symp analog}, we show that this algebraic symplectic reduction is also an open subset of the central fibre of our 1-parameter family and provides an algebraic symplectic analogue of a dense open subset of the additive quotient of the affine space. Conveniently, this also circumvents the need for the ring of invariant functions on this zero level set to be finitely generated.

A fruitful idea of non-reductive geometric invariant theory is to embed a non-reductive group in a reductive group and use techniques of reductive GIT \cite{dorankirwan}. We can embed $\GG_a$ in $\SL_2$ as the upper triangular unipotent subgroup. Over the complex numbers, any linear action of the additive group extends to $\SL_2$ and the ring of $\GG_a$-invariants on the affine space $V$ is naturally isomorphic to the ring of $\SL_2$-invariants on $W:=V \times \AA^2$ and so, in particular, is finitely generated (this is Weitzenb\"{o}ck's Theorem \cite{weitzenbock}). Furthermore, we have an isomorphism
\[  V/\!/\GG_a:= \spec \CC[V]^{\GG_a}  \cong \spec \CC[W]^{\SL_2}=: W/\!/ \SL_2. \]
Theorem \ref{thm biratl sympl} gives a direct comparison between the algebraic symplectic structures on the additive algebraic symplectic analogue of $V/\!/\GG_a$ and the reductive algebraic symplectic analogue of $W/\!/\SL_2$, realised as a birational symplectomorphism differing in codimension at least two.  
Then, via the Kempf-Ness Theorem for $\SL_2$, we establish a hyperk\"{a}hler structure on a dense open subset of the additive algebraic symplectic analogue, in a way similar to the implosion picture. 
We will discuss wider classes of unipotent actions, and the relation with torus actions, as well as details of the hyperk\"{a}hler metric, in follow-up work.  

The outline of this paper is as follows. Section 2 is a summary of the reductive quotient theory, detailing algebraic moment maps for cotangent lifts of the action, symplectic and hyperk\"{a}hler reduction and the results of Proudfoot on hyperk\"{a}hler analogues.  Section 3 addresses the additive case, describing the moment map and the geometry and stability properties of the level sets.  We study the non-zero level set case and the zero level set case separately.  The main results are Theorems \ref{thm fg nonzero} and \ref{1-param family red} describing the algebraic symplectic reductions for non-zero values in a one-parameter family, followed by Proposition \ref{nonred alg symp analog} presenting an algebraic symplectic reduction of a dense open subset of the zero level set, determined by stability analysis.  Section 4 explicitly compares the additive analogue to the familiar $\SL_2$-analogue, summarised in Theorem \ref{thm biratl sympl}, and as a corollary of this result, we obtain a hyperk\"{a}hler structure on a dense open subset of the additive analogue.  In Section 5, we show how the whole story relates to projective geometry, by looking at the cotangent bundle to the blow-up of $n$ points on $\mathbb{P}^{n-2}$.  Section 6 is devoted to examples of computations of rings of additive invariants on the zero level set of the moment map.    

\subsection*{Acknowledgements}  We thank Tamas Hausel, and independently Nicholas Proudfoot, for suggesting the study of $\overline{M}_{0,n}$, and for subsequent conversations.  We are grateful to Andrew Dancer and also to Frances Kirwan for helpful discussions, especially in regard to hyperk\"{a}hler implosion.  The first author was partially supported by Swiss National Science Foundation Award $200021{\_}138071$.  The second author was supported by the Freie Universit\"{a}t Berlin within the Excellence Initiative of the German Research Foundation.

\section{Algebraic symplectic analogues of reductive quotients}

Let $G$ be a reductive group acting linearly on affine space $V \cong \AA^N$; i.e., via a representation $\rho : G \ra \GL(V)$. We can 
construct a GIT quotient of the $G$-action on $V$ by linearising the action on the 
structure sheaf of $V$ using a character $\chi : G \ra \GG_m$; the corresponding GIT 
quotient is a morphism of algebraic varieties
\[ V^{\chi-ss} \ra V/\!/_\chi G\]
where $V^{\chi-ss}$ is the open set of semistable points in $V$ for the $G$-action 
linearised by $\chi$ in the sense of Mumford \cite{mumford}. The GIT quotient 
$V/\!/_\chi G$ is the projective spectrum of the ring of $\chi$-semi-invariants on $V$ 
and is projective over the affine 
GIT quotient $V/\!/G:=\spec \CC[V]^G$.

In this section, we study the lift of the $G$-action to the cotangent bundle $T^*V$, which is 
algebraic symplectic with respect to the Liouville form on $T^*V$ and has an algebraic 
moment map $\mu : T^*V \ra \fg^*$. An algebraic symplectic reduction of the $G$-action on 
$T^*V$ is given by taking a GIT quotient of a level of this moment map at a central value of the co-Lie algebra. 
Finally, we give a survey of Proudfoot's construction of an algebraic symplectic analogue of $V/\!/_\chi G$, 
which, over the complex numbers, is homeomorphic to a hyperk\"{a}hler analogue \cite{proudfootphd}.

\subsection{Algebraic moment maps for cotangent lifts of linear actions}

We recall that an algebraic symplectic manifold $(M,\omega)$ over $\CC$ is a smooth 
complex algebraic variety with an algebraic symplectic form $\omega$ (that is, $\omega$ 
is a non-degenerate closed algebraic 2-form on $M$). This definition also makes sense over a 
more general field $k$. 
The classical example of an algebraic symplectic 
manifold is the cotangent bundle $T^*X$ of a smooth algebraic 
variety $X$ with symplectic form given by the Liouville 2-form. 
For an affine space $V \cong \AA^n$, if we identify $T^*V \cong V \times V^*$ and take 
coordinates $x_1, \dots , x_n$ on $V$ and $\alpha_1,\dots ,\alpha_n$ on $V^*$, then 
the Liouville form on $T^*V$ is given by
\[ \omega = \sum_{l=1}^n dx_l \wedge d \alpha_l.\]

Let $G$ be an affine algebraic group which acts linearly on 
an affine space $V \cong \AA^n$ via a representation $\rho : G \ra \GL(V)$. 
The cotangent lift of the $G$-action to $T^*V \cong V \times V^*$ is given by
\[ g \cdot (v,\gamma) = (\rho(g)v,\gamma \rho(g)^{-1})\]
where $g \in G$ and we consider $v \in V$ as a column vector and $\gamma \in V^*$ as a 
row vector. The following lemma is well-known, but as we were unable to find a suitable 
reference, we include a proof for the sake of completeness.

\begin{lemma}\label{alg mmap def}
Let $G$ be an algebraic group and $\rho: G \ra \GL(V)$ be a $G$-representation. Let $\xi \in \fg^*$ be a point which is fixed by the coadjoint action. Then the following statements hold.
\begin{enumerate}[(i)]
\item The cotangent lift of the $G$-action on $T^*V$ is algebraic symplectic for the 
Liouville form $\omega$ on $T^*V$.
\item This action has an algebraic 
moment map $\mu : T^*V \ra \fg^*$ 
(that is, a $G$-equivariant algebraic map that lifts the infinitesimal action) 
given by
\[ \mu(v,\gamma) \cdot A := \gamma(A_v) = \gamma(\rho_*(A)v)\]
for $A \in \fg^* = (\Lie G)^*$ and $(v,\gamma) \in T^*V \cong V \times V^*$.
\item A point $p \in \mu^{-1}(\xi)$ is a smooth point of the level set $\mu^{-1}(\xi)$ if and only if the $G$-action on $p$ is locally trivial (that is, the stabiliser group $G_p$ is finite).
\item If $p \in \mu^{-1}(\xi)$ is a smooth point, then there is a canonical non-degenerate bilinear form on the vector space
\[ T_p(\mu^{-1}(\xi))/\fg\]
induced by the non-degenerate bilinear form $\omega_p$ on $T_p(T^*V)$.
\end{enumerate} 
\end{lemma}
\begin{proof}
For (i), we need to show that $\sigma_g^* \omega = \omega$, where $\sigma_g : T^*V \ra T^*V$ 
denotes the action of $g \in G$. 
For $p=(v,\gamma) \in T^*V$, we identify $T_p(T^*V) \cong T^*V \cong V \times V^*$ and write 
tangent vectors as pairs $p'=(v',\gamma')$. Then
\begin{align*}
(\sigma^*_g\omega)_p(p',p'') & := \omega_{g \cdot p} (d_p\sigma_g(p'), d_p\sigma_g(p''))  
= \gamma''( \rho(g)^{-1}\rho(g)v') - \gamma' (\rho(g)^{-1}\rho(g)v'') = \omega_p(p',p'').
\end{align*}
For (ii), we first check the equivariance of $\mu$: for $g \in G$, $A \in \fg$ and 
$(v,\gamma) \in T^*V $, 
we have
\begin{align*}
\mu( g \cdot (v,\gamma)) \cdot A & = \mu( (\rho(g)v,\gamma \rho(g)^{-1}) ) \cdot A 
= \gamma \rho(g)^{-1}\rho_*(A)\rho(g)v \\
&= \gamma \text{ad}_{\rho(g)^{-1}}(\rho_*(A)) v = \text{ad}_g^*(\mu(v,\gamma)) \cdot A.
\end{align*}
It remains to check that the moment map lifts the infinitesimal action. 
The infinitesimal action of $A \in \fg$ on $p=(v,\gamma) \in T^*V$ is 
$A_p :=(\rho_*(A)v,\gamma \rho_*(-A)) \in T_p(T^*V)$. Then
\begin{align}\label{mmap inf}
\omega_p(A_p,p') = \gamma'\rho_*(A)v - \gamma \rho_*(-A)v' 
= \gamma'\rho_*(A)v + \gamma \rho_*(A)v'= d_p\mu(p') \cdot A
\end{align}
as required.

For (iii), $p$ is a smooth point of $\mu^{-1}(\xi)$ if and only if $d_p\mu$ 
is surjective. By the infinitesimal lifting property (\ref{mmap inf}) of the 
moment map, we see that $p$ is smooth if and only if 
$\fg_p = 0$ (or, equivalently, if and only if $G_p$ is finite).

Statement (iv) follow from an algebraic version of the Marsden--Weinstein--Meyer Theorem \cite{mw,meyer}. 
Let us briefly sketch the proof. By the preimage theorem, we have $T_p\mu^{-1}(\xi) \subset \text{ker} d_p\mu$. 
Moreover, 
the vector spaces $\text{ker} d_p \mu$ and $T_p(G \cdot p)$ are orthogonal 
complements with respect to the bilinear form $\omega_p$, as
\[ \zeta \in \text{ker} d_p \mu \iff dp\mu(\zeta) \cdot A = 0 \: \forall A \in \fg 
\iff \omega_p(A_p,\zeta) =0 \: \forall A \in \fg \iff \zeta \in T_p(G \cdot p)^{\omega_p}.\]
Therefore, 
$T_p(G \cdot p) \subset T_p(\mu^{-1}(\xi)) = \text{ker} d_p \mu = T_p(G \cdot p)^{\omega_p}$; 
that is $T_p(G \cdot p)$ is an isotropic subspace of $(T_p(T^*V),\omega_p)$ and there is an induced 
non-degenerate bilinear form $\omega'_p$ on 
\[T_p(G \cdot p)^{\omega_p}/T_p(G \cdot p) \cong T_p(\mu^{-1}(\xi))/\fg\] 
defined by $\omega'_p([\zeta_1],[\zeta_2]):= \omega_p(\zeta_1,\zeta_2)$, where the above isomorphism 
uses the infinitesimal action $\fg \ra T_p(T^*V)$, whose image is $T_p(G \cdot p)$ and whose 
kernel $\fg_p$ is zero, as $p$ is a 
smooth point). 
\end{proof}

The level sets of an algebraic moment map $\mu : T^*V \ra \fg^*$ are closed 
algebraic subvarieties of $T^*V$, and so are affine algebraic varieties, 
which are smooth if the level set is taken at a regular value of $\mu$ 
(that is, if and only if $G$ acts locally freely on this level set). Moreover, 
for central elements $\xi \in \fg^*$, the level sets $\mu^{-1}(\xi)$ 
are invariant under the linear $G$-action. 

\subsection{Algebraic symplectic reductions for reductive groups}

For the moment, we restrict ourselves to reductive groups $G$ in order to take a reductive GIT 
quotient of the $G$-action on the level sets of the moment map for the cotangent lift 
of this action. In the following section, we will consider non-reductive groups.

\begin{defn}
Let $G$ be a reductive group acting linearly on $V$ and algebraically symplectically 
on $(T^*V,\omega)$ with algebraic moment map $\mu : T^*V \ra \fg^*$ as above. 
Let $\chi$ be a character of $G$ and $\xi \in \fg^*$ be a central element; 
then we refer to $(\chi,\xi)$ as a central pair. 
The algebraic symplectic reduction of the $G$-action on $(T^*V,\omega)$ at a 
central pair $(\chi,\xi)$ is the GIT quotient of the linear $G$-action on 
the affine algebraic variety $\mu^{-1}(\xi)$ with respect to the 
character defined by $\chi$, we denote this as follows
\[ \pi=\pi_{(\chi,\xi)}: \mu^{-1}(\xi)^{\chi-ss} \ra 
\mu^{-1}(\xi)/\!/_{\chi} G.\]
\end{defn}

\begin{rmk}
Over the complex numbers, for any $\xi \in \fg^*$, one can construct a 
holomorphic symplectic reduction of 
the $G$-action on $T^*V$ at $\xi$ as the topological quotient $\mu^{-1}(\xi)/G$. 
If $\xi $ is a regular value of $\mu$ and the $G$-action on 
$\mu^{-1}(\xi)$ is free and proper, 
then $\mu^{-1}(\xi)/G$ is a holomorphic symplectic 
manifold by a holomorphic version of the Marsden--Weinstein--Meyer Theorem. In 
a more general setting, one can defined a stratified holomorphic symplectic 
structure along the lines of \cite{lermansjamaar}. 
In this case, the holomorphic reduction and the
algebraic symplectic reduction at the trivial character are related by the natural map
\[ \mu^{-1}(\xi)/G \ra \mu^{-1}(\xi)/\!/G\]
which is surjective, but not necessarily injective. The algebraic and holomorphic 
reductions coincide if all points of $\mu^{-1}(\xi)$ are stable.
\end{rmk}

We now state an algebraic version of the Marsden--Weinstein--Meyer Theorem.

\begin{prop}\label{alg mww}
Let $G$ be a reductive group acting linearly on an affine space $V$ 
and consider the algebraic symplectic reduction $\mu^{-1}(\xi)/\!/_{\chi} G$ 
of the $G$-action on $T^*V$ at a central pair $(\chi,\xi)$. 
If $G$ acts (set theoretically) freely on $\mu^{-1}(\xi)^{\chi-ss}$, 
then $\mu^{-1}(\xi)/\!/_{\chi} G$ has a unique structure of a algebraic symplectic 
manifold compatible with that of $T^*V$.
\end{prop}
\begin{proof}
The assumption that $G$ acts on the GIT semistable set 
$\mu^{-1}(\xi)^{\chi-ss}$ with trivial stabiliser groups, implies that 
$\mu^{-1}(\xi)^{\chi-s}= \mu^{-1}(\xi)^{\chi-ss}$, as all orbits 
have top dimension and so must be closed. Hence, the $G$-action on the 
GIT semistable is scheme theoretically free (i.e., set theoretically free 
and proper) and, moreover, the semistable set is contained 
in the smooth locus of $\mu^{-1}(\xi)$ by Lemma \ref{alg mmap def}. 
The GIT quotient 
\[ \pi: \mu^{-1}(\xi)^{\chi-ss} \ra \mu^{-1}(\xi)/\!/_{\chi} G\]
is a geometric quotient, which is smooth and, in fact, a principal 
$G$-bundle by Luna's \'{e}tale slice theorem \cite{luna}.

To prove the existence and uniqueness of the algebraic symplectic form on 
$\mu^{-1}(\xi)/\!/_{\chi} G$, we can use an algebraic version of the 
Marsden--Weinstein--Meyer Theorem: by Lemma \ref{alg mmap def} the 
Liouville form $\omega$ determines a non-degenerate two form on 
$\mu^{-1}(\xi)/\!/_{\chi} G$, which is closed, as $\omega$ is closed.  
\end{proof}

For semisimple groups, such as $\SL_n$, we only have the trivial character 
$\chi = 0$ and the trivial central value $\xi = 0$. 
For the central pair $(0,0)$, the hypothesis of this proposition are 
never satisfied, as all points are semistable and  
$0 \in \mu^{-1}(0)$ is a fixed point of the linear $G$-action. 
In cases when the hypothesis of this proposition fail, we cannot expect to have a 
stratified symplectic structure analogous to \cite{lermansjamaar}, as the GIT quotient 
can identify orbits of different dimensions and the stratification in loc. cit. is 
obtained using the conjugacy class of the stabiliser subgroups. 

Instead, we show there is an open subvariety of the 
algebraic symplectic reduction which admits the structure of an algebraic symplectic 
manifold. Since the GIT quotient restricts to a geometric quotient on the stable locus, 
one may expect to use the stable locus as this open subset. However, we cannot take 
the whole stable locus as the presence of non-trivial but finite stabiliser groups 
can cause singularities in the geometric quotient (for example, the Kleinian surface 
singularities have a GIT quotient construction given by a finite group acting linearly on 
$\AA^2$, for which all points are stable). 

\begin{lemma}\label{alg mmw open}
Let $G$ be a reductive group acting linearly on $V= \AA^n$ and $(\chi,\xi)$ 
be a central pair. Then there exists an open 
subvariety of $\mu^{-1}(\xi)/\!/_{\chi} G$ admitting a unique structure of an 
algebraic symplectic manifold compatible with $T^*V$.
\end{lemma}
\begin{proof}
Let $\pi : \mu^{-1}(\xi)^{\chi-ss} \ra \mu^{-1}(\xi)/\!/_{\chi} G$ denote the GIT quotient, 
which restricts to a geometric quotient on the stable locus 
$\mu^{-1}(\xi)^{\chi-s} \ra \mu^{-1}(\xi)^{\chi-s}/G$. Then the smooth locus 
of $\mu^{-1}(\xi)^{\chi-s}/G$ inherits an algebraic symplectic structure 
following Lemma \ref{alg mmap def}.
\end{proof}

\subsection{The algebraic symplectic analogue}

The inclusion $V \subset T^*V$, as the 
zero section, realises $V$ as a Lagrangian subvariety of $T^*V$. In fact, the 
zero section $V$ is contained entirely in the zero level set of the algebraic moment map and, as $G$ is reductive, the $G$-equivariant closed immersion 
$j: V \hookrightarrow \mu^{-1}(0)$ induces a closed immersion 
\[  \overline{j} : V/\!/_{\chi} G \hookrightarrow \mu^{-1}(0)/\!/_{\chi} G.\]
If the $G$-action on $\mu^{-1}(0)^{\chi-ss}$ is free, then $\mu^{-1}(0)/\!/_{\chi} G$ 
is a smooth algebraic symplectic variety and $\overline{j}$ is Lagrangian, as $V \subset T^*V$ is Lagrangian. 

\begin{defn}
The algebraic symplectic analogue of $V/\!/_\chi G$ is the algebraic 
symplectic reduction of the $G$-action on $T^*V$ at the central pair $(\chi,0)$.
\end{defn}

The above terminology comes from Proudfoot's notion of a hyperk\"{a}hler 
analogue \cite{proudfootphd}. We rephrase 
Proposition 2.4 of loc. cit. in terms of the algebraic symplectic structure.

\begin{prop}[\cite{proudfootphd}, Proposition 2.4]\label{Lagrangian prop}
Let $G$ be a reductive group acting linearly on an affine space $V$ with respect to 
a character $\chi : G \ra \GG_m$. If the $G$-action on $\mu^{-1}(0)^{\chi-ss}$ 
is free and $V^{\chi-ss}$ is non-empty, then
\begin{enumerate}[i)]
\item $V/\!/_{\chi} G$ is an Lagrangian submanifold of $\mu^{-1}(0)/\!/_{\chi} G$, and
\item $T^*(V/\!/_\chi G)$ is a dense open subset of $\mu^{-1}(0)/\!/_{\chi} G$.
\end{enumerate}
\end{prop}

More generally, one has the following result.

\begin{cor}\label{cor alg symp anal open set}
Let $G$ be a reductive group acting linearly on an affine space $V$ 
and $\chi$ be a character of $G$. Then there exist open subvarieties  
$U \subset V/\!/_{\chi} G$ and $U' \subset \mu^{-1}(0)/\!/_{\chi} G$ 
with a commutative square
 \begin{equation*} 
 \xymatrix@1{U  \ar[r]^{} \ar[d]^{} & 
 U' \ar[d]^{\pi}
   \\V/\!/_{\chi} G \ar[r]_{\overline{\iota}} & \mu^{-1}(0)/\!/_{\chi} G }
 \end{equation*}
whose horizontal maps are closed immersions such that
\begin{enumerate}[(i)]
\item $U'$ is an algebraic symplectic manifold;
\item $U$ is a Lagrangian subvariety of $U'$;
\item $T^*U$ is an open set of $U'$, which is dense if non-empty.
\end{enumerate}
\end{cor}
\begin{proof}
Let $U \subset V^{\chi-s}/G$ and $U' \subset \mu^{-1}(0)^{\chi-s}/G$ denote the 
smooth loci of the geometric quotients of the stable sets. Then $U'$ is an algebraic symplectic 
manifold by Lemma \ref{alg mmw open} and the remaining parts follow in 
the same way as in the proof of Proposition \ref{Lagrangian prop}.
\end{proof}

\subsection{Hyperk\"{a}hler reduction}

Over $k = \CC$, the above algebraic symplectic quotient constructions are related to 
a hyperk\"{a}hler quotient construction. The starting point is the observation 
that the cotangent bundle of a complex affine space $V=\AA^n_\CC$ is not only 
algebraic symplectic: the holomorphic symplectic form $\omega =\omega_\CC$ is part 
of a hyperk\"{a}hler structure on $T^*V$ coming from an identification 
$T^*V \cong \HH^n$ where $\HH$ denotes the quaternions. 

Let $1,i,j,k$ denote the standard basis of the quaternions $\HH$; 
then we recall that the standard flat hyperk\"{a}hler structure on $\HH^n$ comes from 
the Quaternionic Hermitian inner product $Q : \HH^n \times \HH^n \ra \HH$ 
given by
\[ Q(z,w) = \sum_{l=1}^n z_lw_l^\dagger\]
where $\dagger$ denotes the Quaternionic conjugate. The triple of complex structures 
$I$,$J$,$K$ on $\HH^n$ are given by left multiplication by $i,j,k$ and the 
hyperk\"{a}hler metric $g$ and corresponding k\"{a}hler forms 
$\omega_I,\omega_j,\omega_K$ are given by the following formula
\[ Q = g + i \omega_I + j \omega_J + k \omega_K.\]
We identify $T^*V \cong \HH^n$ by $ (v,\gamma) \mapsto (v_1 - \gamma_1 j, \dots , v_n - \gamma_n j)$ 
and write the corresponding hyperk\"{a}hler structure on $T^*V$ as $(g,\omega_I,\omega_J,\omega_K)$. 
If we take coordinates $x_1, \dots , x_n$ on $V$ and $\alpha_1,\dots ,\alpha_n$ on 
$V^*$, then
\[ \omega_\RR = \frac{1}{2i} \left( \sum_{k=1}^n dx_k \wedge d \overline{x_k} 
+ d\alpha_k \wedge d \overline{\alpha_k} \right) \quad \text{and} \quad \omega_\CC
 = \sum_{k=1}^n dx_k \wedge d \alpha_k.\]

Let $G$ be a complex reductive group with maximal compact subgroup $K$; 
then $G$ is the complexification of $K$. Suppose that 
we have a linear $G$-action on $V$ such that $K$ acts unitarily. Then we can consider 
the cotangent lift of the $K$-action to $T^*V$. In this case, the triple of K\"{a}hler  
forms are preserved by the action of $K$ on $T^*V$ and, moreover, this $K$-action on 
$T^*V$ is hyper-Hamiltonian; that is, it has a hyperk\"{a}hler moment map 
given by a pair of moment maps $\mu_{HK}=(\mu_\RR,\mu_\CC)$ for the symplectic forms 
$(\omega_\RR := \omega_I,\omega_\CC:= \omega_J + i \omega_K)$. More precisely,
\begin{align*}
\mu_\RR : & \: T^*V \ra \fk^* & & \mu_\RR(v,\gamma) :=  \mu(v) - \mu(\gamma) & \\
\mu_\CC:& \: T^*V \ra \fg^* = \fk_\CC^* & &\mu_\CC(v,\gamma) \cdot A := \gamma(A_v).&
\end{align*}

\begin{rmk}
The action of the maximal compact group $K$ is hyper-Hamiltonian, 
not the action of $G$. For non-reductive groups, there is no longer a bijective 
correspondence with compact groups; thus, it is not immediate how one 
should construct a theory of hyperk\"{a}hler reductions for non-reductive groups. 
\end{rmk}

Let $(\chi_\RR,\chi_\CC)$ be a regular value of the hyperk\"{a}hler moment map 
that is fixed by the coadjoint action. Then the hyperk\"{a}hler reduction of the 
$K$-action on $T^*V$ at $(\chi_\RR,\chi_\CC)$ is defined to be
\[ T^*V/\!/\!/\!/_{(\chi_\RR,\chi_\CC)} K := \mu^{-1}_{HK}((\chi_\RR,\chi_\CC))/K 
= \left( \mu_\RR^{-1}(\chi_\RR) \cap \mu^{-1}_\CC(\chi_\CC) \right)/K.   \]
If the $K$-action on the level set $\mu^{-1}_{HK}((\chi_\RR,\chi_\CC))$ is free, then 
this hyperk\"{a}hler reduction has a unique hyperk\"{a}hler structure compatible with 
the hyperk\"{a}hler structure on the cotangent bundle by \cite{hitchinetal}.  

The hyperk\"{a}hler reduction is the (smooth, rather than algebraic) symplectic 
reduction of the $K$-action on $\mu_\CC^{-1}(\chi_\CC)$ with respect to $\chi_\RR$. 
Hence, by work of Kempf and Ness \cite{kempf_ness}, the hyperk\"{a}hler reduction is 
homeomorphic to the GIT quotient of $G = K_\CC$ acting on $\mu^{-1}_\CC(\chi_\CC)$ with 
respect to the character of $G$ corresponding to $\chi_\RR$ (that is, the algebraic symplectic 
reduction of $G$ acting on $T^*V$ with respect to the central pair $(\chi_\RR,\chi_\CC)$).

Proudfoot proves the following result using the Kempf--Ness homeomorphism and 
the argument outlined in Proposition \ref{Lagrangian prop}.

\begin{prop}[\cite{proudfootphd}, Proposition 2.4]\label{proudfoot1}
Let $G$ be a complex reductive group acting linearly on $V = \AA^n_\CC$ and let 
$\chi \in \fk^*$ be fixed by the coadjoint action. If $\chi$ is a regular value of 
the moment map $\mu$ and $(\chi,0)$ is a regular value of the hyperk\"{a}hler moment 
map $\mu_{HK}$, then $T^*(V/\!/_\chi G)$ is homeomorphic to an open subset of the 
hyperk\"{a}hler reduction $T^*V/\!/\!/\!/_{\chi} K $ which is dense if non-empty.
\end{prop}

\section{Algebraic symplectic analogues of quotients of $\GG_a$-representations}\label{sec G_a asr}

In this section, we study linear actions of the additive group $\GG_a = \CC^+$ on 
an affine space $V := \AA^n$ over the complex numbers. We restrict to the 
field of complex numbers $k = \CC$, as 
in this case, the ring of invariants 
$\CC[V]^{\GG_a}$ is finitely generated. 
In this section, we describe the construction of an algebraic symplectic 
analogue of an open subset of $V/\!/\GG_a$.

\subsection{Non-reductive GIT for representations of the additive group}\label{sec nonred Ga}

Let $\GG_a$ act linearly on $V \cong \AA^n$ over the complex numbers. A key idea of 
non-reductive GIT is to transfer a non-reductive GIT problem to a reductive setting 
by embedding the non-reductive group in a reductive group; see \cite{dorankirwan}. 
We fix a closed embedding $\GG_a \hookrightarrow \SL_2$ as the unipotent subgroup 
in the upper triangular Borel subgroup of $\SL_2$. In this case, the representation 
$\rho : \GG_a \ra \GL(V)$ factors via the embedding $\GG_a \hookrightarrow \SL_2$ by 
the Jacobson--Morozov theorem (alternatively, we can extend the action to $\SL_2$ by 
putting the derivation of the $\GG_a$-action in Jordan normal form). Then 
\[ \CC[V]^{\GG_a} \cong \CC[V \times_{\GG_a} \SL_2]^{\SL_2} \cong 
\CC[V \times \SL_2/\GG_a]^{\SL_2}\]
and as, moreover, $\SL_2/\GG_a \cong \AA^2 - \{ 0 \}$ has a normal affine 
completion $\AA^2$ with complement of codimension 2, it follows that 
\[ \CC[V]^{\GG_a} \cong \CC[W]^{\SL_2},\]
where $W:= V \times \AA^2$. In particular,  
$ \CC[V]^{\GG_a}$ is finitely generated.

\begin{defn}
For a linear $\GG_a$-action on $V \cong \AA^n$, the non-reductive enveloping quotient
is the morphism of affine varieties 
\[ q : V \ra V/\!/\GG_a := \spec (\CC[V]^{\GG_a}). \]
\end{defn}

The term enveloping quotient was introduced in \cite{dorankirwan} to emphasise the fact that 
the morphism $q$ is not surjective in general (for example, the enveloping quotient 
for the $\GG_a$-representation $V=\sym^1(\AA^2)^{\oplus 2}$ is not surjective). A further difference  
of non-reductive actions is that taking invariants is not exact and so, for an invariant 
closed subset $Z$ of $V$, there may be intrinsic invariants to $Z$ which do not extend 
to invariants on $V$. For a detailed comparison of reductive and non-reductive GIT, see \cite{dorankirwan}. 

The enveloping quotient restricts to a better 
behaved quotient on certain stable sets defined in \cite{dorankirwan}. To define these sets, consider the $\GG_a$-equivariant embeddings
\[ \begin{array}{rccccccc} 
i': & V & \begin{smallmatrix} i \\ \ra \end{smallmatrix}& V \times_{\GG_a} \SL_2 & \ra & W = V \times \AA^2 & \ra & \PP:=\PP(\AA^1 \times W) \\
 &v & \mapsto & [v,I] & \mapsto & (v,(1,0)) & \mapsto & [1:v:1:0],
\end{array} \]
where $\SL_2$ acts trivially on the copy of $\AA^1$ appearing in $\PP$. 
Since $\SL_2$ has no non-trivial characters, the line bundles $\cO_{V \times_{\GG_a} \SL_2}$ and 
$\cO_{\PP}(1)$ admit canonical $\SL_2$-linearisations and we can consider 
the (semi)stable sets for the reductive group $\SL_2$ in the sense of Mumford. 
In the terminology of \cite{dorankirwan} Definition 5.2.5, the pair $(\PP,\cO_\PP(1))$ 
is a fine reductive envelope for the linear $\GG_a$-action on $V$.

\begin{defn}[\cite{dorankirwan}, Definitions 5.1.6 and 5.2.11]\label{defn nonred ss}
For $v \in V$, we make the following definitions.
\begin{enumerate}
\item $v$ is Mumford (semi)stable for the $\GG_a$-action on $V$ if $i(v)$ is 
$\SL_2$-(semi)stable in $V \times_{\GG_a} \SL_2$.
\item $v$ is completely (semi)stable for the $\GG_a$-action on $V$ if $i'(v)$ is 
$\SL_2$-(semi)stable in $\PP$.
\end{enumerate}
Let $V^{m(s)s}$ and $V^{\overline{(s)s}}$ denote the Mumford (semi)stable and 
completely (semi)stable sets.
\end{defn} 

Although the definition of the Mumford (semi)stable locus depends a priori on the choice of 
an embedding of $\GG_a$ into a reductive group, it turns out that 
this notion is intrinsic to the $\GG_a$ action (cf. \cite{dorankirwan} 
Proposition 5.1.9). The following theorem relates these concepts.

\begin{thm}[\cite{dorankirwan}, Theorem 5.3.1]
For $\GG_a$ acting linearly on $V =\cong \AA^n$, there is a commutative diagram
 \[\xymatrix{ V^{\overline{s}} \ar@{^{(}->}[r] \ar@{->>}[d]&  
 V^{ms} = V^{mss}  \ar@{^{(}->}[r] \ar@{->>}[d]  & V^{\overline{ss}} = V \ar[d] \\
 V^{\overline{s}}/\GG_a   \ar@{^{(}->}[r] & V^{ms}/\GG_a  \ar@{^{(}->}[r]& V/\!/G
} \]
where all horizontal maps are open inclusions, the two left most vertical maps are 
locally trivial geometric quotients, and the right vertical map is the enveloping quotient.
\end{thm}
 
The completely (semi)stable sets can be computed using the Hilbert--Mumford criterion 
for the $\SL_2$-action on $\PP=\PP(\AA^1 \times V \times \AA^2)$. 
As the $\SL_2$-action on the line $\AA^1$ 
is trivial, we see that $i'(v)=[1:v:1:0]$ is always semistable and so $V^{\overline{ss}} = V$. 
Let $T_{d} \subset \SL_2$ be the diagonal torus; then $T_{d}$ acts with weight zero on $\AA^1$ 
and, as $\SL_2$ acts on $\AA^2$ by its standard representation, the $T_{d}$ weights on 
$\AA^2$ are $(+1,-1)$. Hence, the $T_d$-weight set of $i'(v)$ already contains the weights $0$ and $1$. We state and prove a well-known result that describes the completely stable locus. 

\begin{prop}\label{nonred HM}
Let $\GG_a$ act linearly on $V\cong \AA^n$; then the complement to the completely stable set is the 
non-negative weight space for the diagonal torus $T_d \subset \SL_2$ acting on $V$.
\end{prop}
\begin{proof}
It is straight forward to verify that $\SL_2$-stability of $i'(v)$ is equivalent to $T_d$-stability of $i'(v)$, because $[0:0:1:0]$ has a negative and a positive weight for all tori $T \neq T_d$ which are obtained by conjugating $T_d$ by an element in the upper triangular Borel. By the Hilbert--Mumford criterion, $i'(v)$ is $T_d$-stable if and only if $v \in V$ has a 
negative $T_d$-weight, as $i'(v)$ already has a positive $T_d$-weight.
\end{proof}
 
\subsection{The moment map for the cotangent lift}

Let $\GG_a$ act linearly on $V \cong \AA^n$ over the complex numbers;  then this extends to 
a linear $\SL_2$-action $\rho: \SL_2 \ra \GL(V)$. We can consider the cotangent lifts of 
these actions and their associated moments maps, as well as the cotangent lift of the 
$\SL_2$-action on $W := V \times \AA^2$ using Lemma \ref{alg mmap def}. All these moment 
maps can be described using the moment map $\mu_{\SL_2,V} : T^*V \ra \fs \fl_2^*$ 
for the $\SL_2$-action on $T^*V$.

We identify $T^*V \cong V \times V^*$ and consider points $x\in V$ as 
column vectors and points $\alpha \in V^*$ as row vectors. 
The algebraic moment map $\mu_{\SL_2,V} : T^*V \ra \fs \fl^*_2$ is given by
\begin{equation}\label{cx moment map} 
\mu_{\SL_2,V}(x,\alpha) \cdot B = \alpha(\rho_*(B)x) 
\end{equation}
where $\rho_* : \fs \fl^*_2 \ra \fg \fl(V)$ denotes the derivative of the 
representation $\rho$. Let
\[ H = \left( \begin{array}{cc} 1 & 0 \\ 0 & -1 \end{array} \right) \quad \quad 
E = \left( \begin{array}{cc} 0 & 1 \\ 0 & 0 \end{array} \right) \quad \quad 
F = \left( \begin{array}{cc} 0 & 0 \\ 1 & 0 \end{array} \right)\]
denote the standard basis of $\fs\fl_2$ and let $H^*,E^*,F^*$ denote the dual 
basis of $\fs\fl_2^*$. Then, using this basis, we can write
\[ \mu_{\SL_2,V}(x,\alpha) = \Phi_H(x,\alpha) H^* +  \Phi_E(x,\alpha) E^* +  \Phi_F(x,\alpha)F^*\]
where $\Phi_H, \Phi_E, \Phi_F$ are algebraic functions on $T^*V$. 

\begin{rmk}\label{rmk cx mmap irred}
It is clear from the expression for the  
moment map (\ref{cx moment map}) that if our $\SL_2$-representation $V$ decomposes as a sum of 
$\SL_2$-representations $V = V' \oplus V''$, then the moment map 
on $T^*V$ is the sum of the moment maps for each summand; that is, 
$\mu_{\SL_2, V} = \mu_{\SL_2,V'}+ \mu_{\SL_2,V''}$.
\end{rmk}

In particular, it suffices to compute $\mu_{\SL,V}$ for irreducible 
$\SL_2$-representations $V$.

\begin{ex}\label{ex irred reps}
Let us consider the irreducible $\SL_2$-representation $V_{k+1}:= \sym^k (\CC^2)$ of dimension $k+1$. 
The $\GG_a$-action on $V_{k+1}$ is given by the representation $\GG_a \ra \GL(V_{k+1})$   
\[ c \mapsto \left( \begin{array}{ccccc}  
1 &  \left(\begin{smallmatrix} k \\ 1\end{smallmatrix}\right)c & 
\left(\begin{smallmatrix} k \\ 2 \end{smallmatrix}\right) c^2 & \hdots & c^k \\
0 & 1 & \left(\begin{smallmatrix} k- 1 \\ 1\end{smallmatrix}\right) c &  & c^{k-1} \\
\vdots & 0 & 1 & \ddots & \vdots \\
\vdots & \ddots & 0 & \ddots & c \\
0 & \hdots & \hdots & 0 & 1
\end{array}\right).\]
The derivative of the $\SL_2$-representation gives the infinitesimal action 
$\fs \fl_2 \ra \fg \fl_{k+1}$
\[ \left( \begin{array}{cc} h & e \\f & -h \end{array}\right) \mapsto 
\left( \begin{array}{ccccc}  
kh & ke & 0 & \hdots & 0 \\
f & (k-2)h & (k-1)e & \ddots &  \vdots \\
0 & 2f & \ddots & \ddots & 0 \\
\vdots & \ddots & \ddots & \ddots & e \\
0 & \hdots & 0 & kf & -kh
\end{array}\right).\]

We take coordinates $x=(x_1, \dots, x_{k+1})$ on $V_{k+1}$ and 
$\alpha=(\alpha_1, \dots , \alpha_{k+1})$ on $V_{k+1}^*$ and identify 
$T^*V_{k+1} \cong V_{k+1} \times V_{k+1}^*$. 
Then the cotangent lift of the $\SL_2$-action on $V_{k+1}$ has moment map 
\[ \mu_{\SL_2,V_{k+1}} = \Phi^k_H H^* + \Phi^k_E E^* + \Phi^k_F F^* : T^*V_{k+1} \longrightarrow \fs \fl_2^*\]
where
\[ \Phi^k_H(x,\alpha) = \sum_{j=0}^k (k-2j) x_{j+1}\alpha_{j+1},\]
 \[ \Phi^k_E(x,\alpha) = \sum_{j=1}^k (k+1-j)x_{j+1}\alpha_j \quad \text{and} 
  \quad \Phi^k_F(x,\alpha) = \sum_{j=1}^k j x_j \alpha_{j+1}\]
and the $\GG_a$-action has moment map $\mu_{\GG_a,V_{k+1}}= \Phi^k_E : T^*V_{k+1} \ra \CC$. 
\end{ex}

Since the subgroup $\GG_a \hookrightarrow \SL_2$ corresponds to the 
1-dimensional Lie subalgebra of $\fs \fl_2$ spanned by the basis vector $E$, the algebraic moment 
map $\mu_{\GG_a,V}: T^*V \ra (\Lie \GG_a)^* \cong\CC$ for the $\GG_a$-action on 
$T^*V$ is given by $\Phi_E$. 

\begin{rmk}\label{rmk corr sl2 inv}
Under the isomorphism $\CC[T^*V]^{\GG_a} \cong \CC[T^*V \times \AA^2]^{\SL_2}\cong (\CC[T^*V][u,v])^{\SL_2}$, we note that the $\GG_a$-invariant function $\Phi_E$ corresponds to the $\SL_2$-invariant function
\[ f=u^2 \Phi_E -uv \Phi_H -v^2 \Phi_F.\]
Indeed, $\Phi_E$ is obtained by evaluating $f$ at $(u,v) = (1,0)$ and the $\SL_2$-invariance of $f$ follows from the $\SL_2$-equivariance of $\mu_{\SL_2, V}$.
\end{rmk}

As the $\SL_2$-representation $W$ is the direct sum of $V$ and the standard 
$\SL_2$-representation $\AA^2$, the moment map for the $\SL_2$-action on $T^*W$ 
is the sum of the  moment maps for $\SL_2$-acting on $T^*V$ and $T^*\AA^2$. 
By Example \ref{ex irred reps}, the moment map for the $\SL_2$-action on $T^*\AA^2$ is given by
\[\mu_{\SL_2,\AA^2} (u,v,\lambda,\eta)=(u\lambda -v\eta)H^* + v\lambda E^* + u\eta F^*\]
where $(u,v,\lambda,\eta) \in \AA^2 \times \AA^2)^*$. 
Hence, $\mu_{\SL_2,W} : T^*W \ra \fs \fl^*_2$ is given by
\begin{equation}\label{SL2 mmap W}
 \mu_{\SL_2,W}(z,\gamma) = (\Phi_H(x,\alpha) + u\lambda - v \eta ) H^* +  
(\Phi_E(x,\alpha)+v \lambda )E^* +  (\Phi_F(x,\alpha)+ u \eta  )F^*,
\end{equation}
where $z = (x,(u,v)) \in W =V \times \AA^2$ and 
$\gamma = (\alpha, (\lambda,\eta)) \in W^* \cong V^* \times (\AA^2)^*$.

\begin{rmk}
Let $\iota : T^*V \hookrightarrow T^*W$ be the $\GG_a$-equivariant inclusion 
$(x,\alpha) \mapsto (x,(1,0),\alpha, (0,1))$. Then there is a commutative diagram
\[ \xymatrix{
T^*V \ar[d]_{\iota} \ar[r]^{\mu_{\GG_a,V}} & \CC \\
T^*W \ar[r]_{\mu_{\SL_2,W}} & \fs\fl_2^* \ar[u]}
\]
where $\fs \fl_2^* \ra \CC \cong \text{Lie} \: (\GG_a)^*$ is the projection map induced by the inclusion 
$\GG_a \hookrightarrow \SL_2$.
\end{rmk}

\subsection{The geometry of the level sets of the moment map}

We continue to use the notation above: 
we let $V$ be a complex $\GG_a$-representation and let $W = V \times \AA^2$ be the 
associated $\SL_2$-representation. As above, we write the moment map for the 
$\SL_2$-action on $T^*V$ as
\[\mu_{\SL_2,V} = \Phi_H H^* +  \Phi_E E^* +  \Phi_F F^* \]
and we write the moment map for the $\GG_a$-action on $T^*V$ (resp. $\SL_2$-action on $T^*W$) as 
\[ \mu_{\GG_a}  : T^*V \ra \CC \quad \quad (\text{resp. } \: \: \mu_{\SL_2}  : T^*W \ra \fs \fl_2^*). \]

As every element $\xi \in \CC \cong (\Lie \GG_a)^*$ is central, every level set 
$H_V(\xi):= \mu_{\GG_a}^{-1}(\xi)$ is a $\GG_a$-invariant closed subvariety of $T^*V$. 
Henceforth, we exclude a degenerate case: 
if the $\GG_a$-action on $V$ is trivial, then $\mu_{\GG_a} =0$ and so the level sets are either 
empty (for $\xi \neq 0$) or equal to the whole cotangent bundle $T^*V$ (for $\xi = 0$). 

We first describe the level sets for the $\GG_a$-representations $V_{k+1} = \sym^k(\CC^2)$.

\begin{lemma}\label{lem hypersurfaces}
For $k \geq 1$, let $V_{k+1}:= \sym^k (\CC^2)$ be the irreducible $\SL_2$-representation 
of dimension $k+1$ and consider the moment map $\mu_{\GG_a,V_{k+1}} : T^*V_{k+1} \ra \CC$ for 
the cotangent lift of the $\GG_a$-action. For $\xi \in \CC$, the level set 
$H_k(\xi) := \mu_{\GG_a,V_{k+1}}^{-1}(\xi)$ is a quadric affine hypersurface in $T^*V_{k+1}$. 
Moreover, these hypersurfaces have the following properties.
\begin{enumerate}
\item $H_k(\xi)$ is irreducible if and only if $\xi \neq 0$ or $k > 1$. The hypersurface $H_1(0)$ has two irreducible components, which are both isomorphic to $\AA^3$.
\item $H_k(\xi)$ is non-singular if and only if $\xi \neq 0$.
\item For $\xi = 0$, the singular locus is the $\GG_a$-fixed locus; that is, $\sing H_k(0) =(T^*V_{k+1})^{\GG_a}$.
\item $H_k(0)$ is normal if and only if $k > 1$.
\end{enumerate}
\end{lemma}
\begin{proof}
We use the explicit form of the moment map 
\[\mu_{\GG_a,V_{k+1}}(x,\alpha) = \sum_{j=1}^k (k-j+1)x_{j+1}\alpha_j.\]
given by Example \ref{ex irred reps}. As this is a homogeneous quadratic function, 
the corresponding level sets are affine quadric hypersurfaces. We note that each monomial appearing in this hypersurface equation involves distinct variables and, for $k > 1$, we have more than one such monomial; thus, for $k >1$, the hypersurface $H_k(\xi)$ is irreducible. If $ k =1$, then 
$\mu_{\GG_a,V_2}(x,\alpha) = x_2 \alpha_1$ and the hypersurface $H_1(\xi)$ is irreducible 
if and only if $\xi \neq 0$.

The singular loci of these hypersurfaces are 
\[ \sing H_k(\xi) := \{ (x,\alpha) \in H_k(\xi) : x_2 = \cdots = x_{k+1} = \alpha_1 = \cdots = \alpha_k = 0 \}.\]
In particular, $\sing H_k(\xi)$ is empty for $\xi \neq 0$. Conversely, for $\xi =0$, the singular locus of 
$H_k(0)$ is isomorphic to $\AA^2$ and coincides with the $\GG_a$-fixed locus  
\[ (T^*V_{k+1})^{\GG_a}=\{ (x,\alpha)  : x_2 = \cdots = x_{k+1} = \alpha_1 = \cdots = \alpha_k = 0 \} = \sing H_k(0). \]
Since $\text{codim}({\sing H_k(0)}, H_k(0)) = 2k+1 -2 = 2k - 1$, we see that 
$H_k(0)$ is non-singular in codimension 1 if and only if $ k > 1$. 
In particular, $H_1(0)$ is not normal, as it fails to be non-singular in codimension 1. 
Since normality is equivalent to regular in codimension 1 and Serre's $S_2$ condition, which 
automatically holds for irreducible affine hypersurfaces (for example, 
see \cite{mumfordred} III $\S$8 Proposition 2), we conclude that, for $k >1$, the hypersurfaces 
$H_k(0)$ are normal.
\end{proof}

We can then extend this to any $\GG_a$-representation using Remark \ref{rmk cx mmap irred}. 

\begin{prop}\label{prop hypersurfaces}
Let $V$ be a (non-trivial) $\GG_a$-representation and let $\mu_{\GG_a} : T^*V \ra \CC$ 
denote the moment map for the cotangent lift of this action. 
For $\xi \in \CC$, the level set $H_V(\xi):=\mu_{\GG_a}^{-1}(\xi)$ is an 
affine quadric hypersurface in $T^*V$ with the following properties.
\begin{enumerate}
\item $H_V(\xi)$ is irreducible if and only if $\xi \neq 0$ or 
$V \neq \sym^1 (\CC^2) \oplus \sym^0(\CC^2)^{\oplus n}$, as an $\SL_2$-representation, for some $n \geq 0$.
\item $H_V(\xi)$ is non-singular if and only if $\xi \neq 0$.
\item For $\xi = 0$, the singular locus is the $\GG_a$-fixed locus; that is, $\sing H_V(0) =(T^*V)^{\GG_a}$.
\item $H_V(0)$ is normal if and only if $V \neq \sym^1 (\CC^2) \oplus \sym^0(\CC^2)^{\oplus n}$,
as an $\SL_2$-representation, for some $n \geq 0$.
\end{enumerate}
\end{prop}
\begin{proof}
We extend the $\GG_a$-action on $V$ to $\SL_2$ and 
decompose $V$ as a sum of irreducible $\SL_2$-representations
\[ V = \sym^{k_1}(\CC^2) \oplus \cdots \oplus \sym^{k_m}(\CC^2). \]
On $\sym^{k_j}(\CC^2)$ (resp. $\sym^{k_j}(\CC^2)^*$), 
we take coordinates $x_1^{(j)}, \dots , x_{k_j +1}^{(j)}$ 
(resp. $\alpha_1^{(j)}, \dots , \alpha_{k_j +1}^{(j)}$). 
By Remark \ref{rmk cx mmap irred}, we have
\[ \mu_{\GG_a}(x,\alpha) = \sum_{j=1}^m \mu_{\GG_a,k_j}(x_i^{(j)},\alpha_i^{(j)})
= \sum_{j=1}^m \sum_{i=1}^{k_j} (k_{j} +1 - i)x_{i+1}^{(j)}\alpha_i^{(j)}.\]
Hence
\[ \sing H_V(\xi) := \{ (x,\alpha) \in H_V(\xi) : x_{i+1}^{(j)}= \alpha_i^{(j)} = 0  \: \text{ for }
 \: 1 \leq j \leq m \: \text{ and } \: 1 \leq i \leq k_j \}\]
and we see that (1), (2) and (3) hold. We prove (4) as in Lemma \ref{lem hypersurfaces}: we have 
\[ \text{codim}({\sing H_V(0)}, H_V(0)) =  \sum_{j=1}^m2(k_j+1) -1 - 2m  = 2\sum_{j=1}^m k_j -1 \]
where $k_j \geq 0$, for $j = 1, \dots , m$, and, as we assumed the action was non-trivial, 
there is at least one $j$ such that $k_j>0$. Hence $ \text{codim}({\sing H_V(0)}, H_V(0)) \geq 1$, with 
equality if and only if there exists $j_0$ such that $k_{j_0} =1$ and $k_j = 0$ for all $j \neq j_0$ 
(that is, if and only if $V = \sym^1 (\CC^2) \oplus \sym^0(\CC^2)^{\oplus m-1}$). 
In exactly the same way as in the proof of Lemma \ref{lem hypersurfaces}, $H_V(0)$ is normal if 
and only if $V \neq \sym^1 (\CC^2) \oplus \sym^0(\CC^2)^{\oplus n}$ for some $n \geq 0$.
\end{proof}

\subsection{(Semi)stability on the level sets of the moment map} 
Let $(T^*V)^{\overline{s}}$ be the completely stable set for the $\GG_a$-action and write $H_V(\xi)^{\overline{s}}:=H_V(\xi) \cap (T^*V)^{\overline{s}}$ 
(that is, the completely stable set determined with respect to the ambient invariants on $T^*V$). 

\begin{lemma}\label{lemma comp stable}
Let $V$ be a (non-trivial) $\GG_a$-representation; then fix an extension to $\SL_2$ and write $V$
as a sum of irreducible $\SL_2$-representations
\[ V = \sym^{k_1}(\CC^2) \oplus \cdots \oplus \sym^{k_m}(\CC^2) \]
and take coordinates $x_1^{(j)}, \dots , x_{k_j +1}^{(j)}$ 
(resp. $\alpha_1^{(j)}, \dots , \alpha_{k_j +1}^{(j)}$) on $\sym^{k_j}(\CC^2)$ 
(resp. $\sym^{k_j}(\CC^2)^*$). Then the complement to the completely stable locus is given by  
\[ T^*V - (T^*V)^{\overline{s}} = \{(x,\alpha) : \alpha_i^{(j)} = x_{k_j+2 -i}^{(j)}=0 \: \:
\mathrm{ for } \: \: 1 \leq j \leq m \: \: \mathrm{ and } \: \: 1 \leq i \leq \lceil {k_j}/{2} \rceil \}. \]
\end{lemma}
\begin{proof}
This is a direct consequence of Proposition \ref{nonred HM} 
and the fact that the weights for the diagonal torus $T_d \subset \SL_2$ acting on 
$\sym^{k}(\CC^2)$ are $k, k-2, \dots , -k+2, -k$ and, dually, the $T_d$-weights on 
$\sym^k(\CC^2)^*$ are $-k,-k+2, \dots, k-2,k$.
\end{proof}

An important consequence of this is the following result.

\begin{cor}\label{cor stab hyp}
Let $V$ be a $\GG_a$-representation with moment map 
$\mu_{\GG_a} : T^*V \ra \CC$; then 
\[\mu_{\GG_a}(T^*V - (T^*V)^{\overline{s}}) = 0.\] 
In particular, if $\mu_{\GG_a}(x,\alpha) \neq 0$ then $(x,\alpha)$ is completely stable 
and so, for $\xi \neq 0$, the hypersurfaces $H_V(\xi):= \mu_{\GG_a}^{-1}(\xi)$ are 
completely stable everywhere; that is, $H_V(\xi)^{\overline{s}}=H_V(\xi)$. Furthermore, 
$H_V(0)^{\overline{s}} \subset H_V(0) - H_V(0)^{\GG_a}$.
\end{cor}
\begin{proof}
From the proof of Proposition \ref{prop hypersurfaces}, the moment map is given by
\[ \mu_{\GG_a}(x,\alpha) =  \sum_{j=1}^m \left( \sum_{i=1}^{\lceil {k_j}/{2} \rceil} 
(k_{j} +1 - i)x_{i+1}^{(j)}\alpha_i^{(j)} +  \sum^{k_j}_{i=\lceil {k_j}/{2} \rceil +1} 
(k_{j} +1 - i)x_{i+1}^{(j)}\alpha_i^{(j)} \right).\]
On $T^*V - (T^*V)^{\overline{s}}$, the first part of this expression vanishes, as the $\alpha_i^{(j)}$s 
are zero for $ 1 \leq i \leq \lceil {k_j}/{2} \rceil$, and the second part vanishes, as the 
$x_{i+1}^{(j)}$s are zero for $\lceil{k_j}/{2} \rceil +1 \leq i \leq k$.
\end{proof}

\subsection{Algebraic symplectic reduction at non-zero values}

In this section, we construct an algebraic symplectic reduction of the linear $\GG_a$-action on $T^*V$ at 
$\xi \in \CC^*$ by taking the $\GG_a$-quotient of the level set $H_V(\xi)=\mu_{\GG_a}^{-1}(\xi)$. We first describe this quotient.

\begin{thm}\label{thm fg nonzero}
Let $\GG_a$ act linearly on an affine space $V$ with moment map $\mu_{\GG_a} : T^*V \ra \CC$. For $\xi \neq 0$, consider the hypersurface $H_V(\xi):= \mu_{\GG_a}^{-1}(\xi)$. Then we have the following:
\begin{enumerate}[\quad i)]
\item $H_V(\xi)^{\overline{s}} = H_V(\xi)^{ms} = H_V(\xi)$,
\item $H_V(\xi)$ admits a locally trivial geometric quotient $ H_V(\xi)/\GG_a$, which is a smooth,
\item $ H_V(\xi)/\GG_a$ is affine and so $\CC[H_V(\xi)]^{\GG_a}$ is finitely generated,
\item all invariants in $\CC[H_V(\xi)]^{\GG_a}$ are restrictions of invariants in $\CC[T^*V]^{\GG_a}$,
\item $H_V(\xi) \ra H_V(\xi)/\GG_a$ is a trivial $\GG_a$-torsor.
\end{enumerate}
In particular, $H_V(\xi)/\GG_a=H_V(\xi)/\!/ \GG_a:=\spec \CC[H_V(\xi)]^{\GG_a}$ and the non-reductive 
quotient $\pi : H_V(\xi) \ra  H_V(\xi)/\!/ \GG_a$ is surjective.
\end{thm}
\begin{proof}
The first statement is Corollary \ref{cor stab hyp} and it follows that there is a Zariski locally trivial geometric quotient $H_V(\xi) \ra H_V(\xi)/\GG_a$. By Proposition \ref{prop hypersurfaces}, $H_V(\xi)$ is smooth and so the quotient $H_V(\xi)/\GG_a$ is also smooth which gives ii). The quotient $H_V(\xi)/\GG_a$ is a priori only quasi-affine; however, 
we will show that $\SL_2 \times_{\GG_a} H_V(\xi)$ is affine and deduce from this that the quotient is affine. Consider the closure 
\[\overline{\SL_2 \times_{\GG_a} H_V(\xi)}\] 
of $\SL_2 \times_{\GG_a} H_V(\xi)$ in $\AA^2 \times T^*V$. By Remark \ref{rmk corr sl2 inv}, this hypersurface is defined 
by the vanishing of $f - \xi$ where
\[ f=u^2 \Phi_E -uv \Phi_H - v^2 \Phi_F.\]
The boundary, which is given by setting $(u,v) = (0,0)$, is empty and so $\SL_2 \times_{\GG_a} H_V(\xi)$ is affine. 
As the reductive GIT-quotient of an affine variety is affine, it follows that 
\[ H_V(\xi)/\GG_a \cong (\SL_2 \times_{\GG_a} H_V(\xi))/\SL_2 \]
is affine. Hence, $H_V(\xi)/\GG_a$ is affine and so has finitely generated coordinate ring 
\[\CC[H_V(\xi)/\GG_a] \cong \CC[H_V(\xi)]^{\GG_a}\cong \CC[\SL_2 \times_{\GG_a} H_V(\xi)]^{\SL_2}.\]
Moreover, we have a surjection 
\[\CC[T^*V]^{\GG_a} \cong \CC[T^*V \times \AA^2]^{\SL_2} \twoheadrightarrow \CC[\SL_2 \times_{\GG_a} H_V(\xi)]^{\SL_2} \cong \CC[H_V(\xi)]^{\GG_a}\]
which completes the proof of iii-iv).

Finally, to show $H_V(\xi) \ra H_V(\xi)/\GG_a$ is isomorphic to the trivial $\GG_a$-torsor, we show there 
exists a function $\sigma \in \CC[H_V(\xi)]$ whose derivation under the $\GG_a$-action is 1, 
which gives a section perpendicular to the $\GG_a$-action that enables us to define an isomorphism 
with the trivial $\GG_a$-torsor over $H_V(\xi)/\GG_a$. More precisely, we claim that there 
exists $\sigma'$ such that
$\mu_{\GG_a}=D(\sigma')$, from which it follows that $\sigma:=\sigma'/\xi$ has derivation $1$ 
in $\CC[H_V(\xi)]$. It suffices to provide $\sigma'$ for the irreducible $\SL_2$-representations 
$V = \sym^n \CC^2$: if we take coordinates $x_i,\alpha_i$ on $T^*V$ as above, then, for appropriate choices of constants $\lambda_i$, the image of $\sigma':= \sum_{i=1}^{n+1} \lambda_i x_i\alpha_i$ under the derivation is $\mu_{\GG_a}$. 
\end{proof}

The following theorem proves that the algebraic symplectic reductions $H_V(\xi)/\GG_a$ are algebraic 
symplectic manifolds, which fit together to give a 1-parameter family of algebraic symplectic reductions. Moreover, the fibres of this 1-parameter family are isomorphic as algebraic varieties, but not as algebraic symplectic varieties; instead, they are \lq scaled symplectomorphic' to each other.

\begin{thm}\label{1-param family red}
For $\xi \neq 0$, the quotient $\mathcal{X}_{\xi}:=H_V(\xi)/\GG_a$ is an algebraic symplectic manifold with symplectic form 
$\omega_\xi$ induced by the Liouville form $\omega$, which form a 1-parameter family
 \begin{equation*}\label{comm lagrang} 
 \xymatrix@1{\mathcal{X}_{\xi} \ar@{^{(}->}[r]^{} \ar[d] & 
 \mathcal{X} \ar[d]
   \\ \quad \xi \quad \ar@{^{(}->}[r] &\quad  \CC^* \quad }
  \end{equation*}
  where $H_V(\CC^*):=\mu_{\GG_a}^{-1}(\CC^*) \ra \mathcal{X}:=H_V(\CC^*)/\GG_a$ is a trivial $\GG_a$-torsor.  
 Furthermore, for $\xi_1,\xi_2 \in \CC^*$, there is an isomorphism 
 \[\Phi_{\xi_1,\xi_2} :H_V(\xi_1)/\GG_a \ra H_V(\xi_2)/\GG_a\]
 of algebraic varieties which scales the symplectic forms as follows: 
 $\xi_2 \omega_{\xi_1} = \xi_1 \Phi^*_{\xi_1,\xi_2}\omega_{\xi_2}$.
 \end{thm}
\begin{proof}
By Theorem \ref{thm fg nonzero}, the quotient $H_V(\xi)/\GG_a$ is a smooth 
affine variety for $\xi \neq 0$ and, 
moreover, it has an algebraic symplectic form $\omega_\xi$ induced 
by the Liouville form $\omega$ on $T^*V$ by Lemma \ref{alg mmap def}. 

The linear action of $\GG_a$ on the affine space $T^*V$ has finitely 
generated ring of invariants and so we can 
consider the non-reductive enveloping quotient $\pi: T^*V \ra T^*V/\!/\GG_a:=\spec\CC[T^*V]^{\GG_a}$, which 
restricts to a Zariski locally trivial geometric quotient on the completely stable locus. By Corollary \ref{cor stab hyp}, 
\[H_V(\CC^*):=\mu_{\GG_a}^{-1}(\CC^*) \subset (T^*V)^{\overline{s}}\] 
and so we have a Zariski locally trivial geometric 
quotient $H_V(\CC^*) \ra H_V(\CC^*)/\GG_a$. The morphism $\mathcal{X}:=H_V(\CC^*)/\GG_a \ra \CC^*$ defining this 1-parameter family is given by the $\GG_a$-invariant 
function $\mu_{\GG_a}$ and restricts to the quotient $H_V(\xi)/\GG_a$ at $\xi \in \CC^*$, as every $\GG_a$-invariant function 
on $H_V(\xi)$ extends to $\GG_a$-invariants on $T^*V$. Moreover, $H_V(\CC^*) \ra H_V(\CC^*)/\GG_a$ is a trivial $\GG_a$-torsor, 
as the derivation of the function $\sigma'/\mu_{\GG_a} \in \CC[H_V(\CC^*)]= \CC[T^*V]_{\mu_{\GG_a}}$ is 1, where $\sigma'$ 
is the function appearing in the proof of Theorem 3.14.

Let $\phi : T^*V \ra T^*V$ be the map given by multiplication by $C=\xi_2/\xi_1$ on $V$ and the identity map on $V^*$;  
this map restricts to a $\GG_a$-equivariant isomorphism
\[ \phi : H_V(\xi_1) \ra H_V(\xi_2)\]
and induces an isomorphism of quotients 
\[\Phi :H_V(\xi_1)/\GG_a \ra H_V(\xi_2)/\GG_a.\]
The pullback of the Liouville form $\omega$ on $T^*V$ via $\phi : T^*V \ra T^*V$ satisfies $\phi^*\omega = C \omega$. Let $\pi_{\xi}: H_v(\xi) \ra H_V(\xi)/\GG_a$ be the quotient map and 
$i_{\xi} :H_V(\xi) \hookrightarrow T^*V$ be the inclusion; then
\[ \pi_{\xi_1}^*{\Phi}^*(\omega_{\xi_2})= \phi^* \pi_{\xi_2}^* (\omega_{\xi_2}) = \phi^* i_{\xi_2}^* (\omega) = i^*_{\xi_1} \phi^*(\omega)=i_{\xi_1}^*(C\omega) = C \pi^*_{\xi_1} (\omega_\xi), \]
and, as $\pi_{\xi_1}$ is surjective, it follows that ${\Phi}^* (\omega_{\xi_2}) = C\omega_{\xi_1}$.
\end{proof}

Following the proof of this result, we are able to describe the central fibre of this 1-parameter family.

\begin{cor}\label{cor central fibre}
The central fibre of the 1-parameter family $\mathcal{X}=H_V(\CC^*)/\GG_a$ constructed above from the algebraic symplectic reductions of the $\GG_a$-action on $T^*V$ at non-zero values is the zero fibre of the morphism $\mu_{\GG_a} : T^*V/\!/\GG_a \ra \CC$. More precisely, there is a $\GG_a$-invariant morphism 
\[ H_V(0) \ra \mathcal{X}_0\cong \spec \CC[T^*V]^{\GG_a}/(\mu_{\GG_a})\]
constructed using the extrinsic $\GG_a$-invariant functions on the ambient space $T^*V$.
\end{cor}

At this point, we should emphasise that in general, the hypersurface $H_V(0) \subset T^*V$ admits $\GG_a$-invariant functions which do not extend to $T^*V$ and so there is a non-surjective homomorphism
\[ \CC[T^*V]^{\GG_a}/(\mu_{\GG_a}) \ra \CC[H_V(0)]^{\GG_a}\]
and, if the latter is finitely generated, then we have a morphism of affine varieties
\[ H_V(0)/\!/\GG_a := \spec  \CC[H_V(0)]^{\GG_a} \ra \mathcal{X}_0. \] 

\subsection{Constructing the algebraic symplectic analogue}

The non-reductive quotient $V/\!/\GG_a$ will often be singular, as not all points in $V$ 
are stable (for example, the origin has full stabiliser group, so is not stable). Therefore, we 
can only expect to provide an algebraic symplectic analogue of an open subset of this quotient. 

\begin{prop}\label{nonred alg symp analog}
Let $V$ be a $\GG_a$-representation. Then the open set $U:=V^{\overline{s}}/\GG_a \subset V/\!/ \GG_a$ 
has an algebraic symplectic analogue $U':=H_V(0)^{\overline{s}}/\GG_a$; that is, we have the following.
\begin{enumerate}[(i)]
\item There is a closed immersion $U \hookrightarrow U'$.
\item $U'$ is an algebraic symplectic manifold, with algebraic symplectic form $\omega'$ induced by the Liouville form $\omega$ on $T^*V$.
\item $U \subset U'$ is a Lagrangian subvariety.
\item $T^*U$ is an open subset of $U'$ which is dense if non-empty.
\item $U'$ is an open subset of the central fibre $\mathcal{X}_0$ of the 1-parameter family $\mathcal{X}$.
\end{enumerate}
\end{prop}
\begin{proof}
As $V^{\overline{s}} \subset V$ is open and the quotient $V^{\overline{s}} \ra U:= V^{\overline{s}}/\GG_a$ 
is a Zariski locally trivial geometric quotient, $U \subset V/\!/\GG_a$ is a smooth open subvariety. The hypersurface $H_V(0)$ is 
singular, but the completely stable locus avoids the singularities:
\[ {H_V}(0)^{\overline{s}} \subset H_V(0) - H_V(0)^{\GG_a}=H_V(0)_{\text{reg}}\]
by Proposition \ref{prop hypersurfaces} and Corollary \ref{cor stab hyp}. Hence, as 
${H_V}(0)^{\overline{s}} \ra {H_V}(0)^{\overline{s}}/\GG_a=:U'$ is a Zariski locally trivial geometric quotient, 
$U'$ is also smooth. Since the completely stable locus is determined by the extrinsic invariants on $T^*V$, it follows that $V \cap H_V(0)^{\overline{s}} =V^{\overline{s}}$ and we have a pullback square 
\[\xymatrix@1{ V^{\overline{s}}\ar[r] \ar[d] &  H_V(0)^{\overline{s}}  \ar[d]
 \\
V  \ar[r] &  H_V(0)},\]
such that both horizontal morphisms are closed immersions. It follows that the induced map on 
$\GG_a$-quotients $U \ra U'$ is also a closed immersion.

The algebraic symplectic structure on $U'$ is a consequence of the algebraic version of the 
Marsden-Weinstein-Meyer Theorem and the fact that $U \subset U'$ is a Lagrangian subvariety 
follows from the fact that $V \subset T^*V$ is a Lagrangian subvariety. The 
proof of the final statement follows exactly as in the proof of Proposition 
\ref{Lagrangian prop}.

For v), we recall that the central fibre $\mathcal{X}_0$ of the 1-parameter family $\mathcal{X}$ is 
the $\GG_a$-quotient of $H_V(0)$ obtained using only the $\GG_a$-invariant functions on $T^*V$. 
Since the completely stable locus $H_V(0)^{\overline{s}} $ is determined using invariant 
functions on $T^*V$, the Zariski locally trivial geometric quotient $H_V(0)^{\overline{s}} \ra U'$ is 
obtained by restricting $H_V(0) \ra \mathcal{X}_0$ and so, in particular, $U'$ is an open subset of 
$\mathcal{X}_0$.
\end{proof}

The completely stable set $H_V(0)^{\overline{s}}$ is determined using only the ambient $\GG_a$-invariants on the affine space $T^*V$. However, there may be additional intrinsic $\GG_a$-invariants on 
$H_V(0)$ which do not extend to invariants on $T^*V$, as the homomorphism
\[\CC[T^*V]^{\GG_a} \ra \CC[H_V(0)]^{\GG_a}\]
is not surjective in general. Since the $\GG_a$-action on $H_V(0)$ does not extend to $\SL_2$, we 
cannot conclude that $\CC[H_V(0)]^{\GG_a}$ is finitely generated (see $\S$\ref{sec fg} below). 

The additional intrinsic invariants can be used to construct a potentially larger open subset which admits a Zariski locally trivial geometric quotient, referred to as the \lq locally trivial stable set' in \cite{dorankirwan}. The $\GG_a$-quotient of this locally trivial stable set $H_V(0)^{lts}$ would also admit the structure of a smooth algebraic symplectic variety. 
However, as there is no Hilbert--Mumford style criteria to compute the locally 
trivial stable set, we would need to compute the invariant ring, which is impractical and thus 
justifies our use of the potentially smaller, but easily understandable, subset of completely stable points. Furthermore, the following lemma shows that these subsets can differ in at most a codimension 2 subset.

\begin{lemma}
The complement of $H_V(0)^{\overline{s}} \subset H_V(0)^{lts}$ has codimension $\geq 2$. 
\end{lemma}
\begin{proof}
The codimension of $H_V(0)^{\overline{s}} \subset H_V(0)^{lts}$ is greater than or equal to the 
codimension of $H_V(0)^{\overline{s}} \subset H_V(0)$, which we can explicitly compute using Lemma \ref{lemma comp stable}. If we write $V$ as a sum of irreducible $\SL_2$-representations $V =   \sym^{k_1}(\CC^2) \oplus \cdots \oplus \sym^{k_m}(\CC^2)$, then the codimension  of $H_V(0)^{\overline{s}} \subset H_V(0)$ is
\[ 2 \left(\sum_{j=1}^m \lceil {k_j}/{2} \rceil \right) -1,\]
which is greater than or equal to 2, unless $V = \sym^1(\CC^2)$ or $V=\sym^2(\CC^2)$. If $V = \sym^1(\CC^2)$, then $H_V(0)$ is 3 dimensional and $H_V(0)^{\overline{s}} = H_V(0)^{lts}$, as the complement to the completely stable locus is the $\GG_a$-fixed locus (see Proposition \ref{prop sym1} below). 
If $V = \sym^2(\CC^2)$, then we also have $H_V(0)^{\overline{s}} = H_V(0)^{lts}$ (essentially, as on 
the complement of the completely stable locus the action is not separated, see $\S$\ref{sec sym2}).
\end{proof}

\subsection{Finite generation of the invariant ring of $H_V(0)$}\label{sec fg}

In the simplest example $V = \sym^1(\CC^2)$, the ring  $\CC[H_V(0)]^{\GG_a}$ is non-finitely generated.

\begin{ex}\label{ex non fg}
Consider the 2-dimensional $\GG_a$-representation $V = \sym^1(\CC^2)$. 
If we take coordinates $x_1,x_2$ on $V$ 
and $\alpha_1,\alpha_2$ on $V^*$, then the action of $s \in \GG_a$ on $T^*V \cong V \times V^*$ is
\[s \cdot (x_1,x_2,\alpha_1,\alpha_2) = (x_1 +s x_2,x_2,\alpha_1,\alpha_2 - s\alpha_1).\] 
By Example \ref{ex irred reps}, the moment map $\mu_{\GG_a} : T^*V \ra \CC$ is given by 
$\mu_{\GG_a}(x,\alpha) = x_2 \alpha_1$. The ring of $\GG_a$-invariant functions on 
the cotangent space is finitely generated 
\[ \CC[T^*V]^{\GG_a}= \CC[x_2,\alpha_1,x_1\alpha_1 + x_2 \alpha_2],\]
but the ring of $\GG_a$-invariant functions on the hypersurface $H_V(0):=\mu_{\GG_a}^{-1}(0)$ is not: 
for $n \geq 0$, 
\[ x_1^n \alpha_1 \quad \text{and} \quad x_2\alpha_2^n \]
are $\GG_a$-invariant functions on $\mu_{\GG_a}^{-1}(0)$, but 
the functions $x_1$ and $\alpha_2$ are not $\GG_a$-invariant. 
We see that taking $\GG_a$-invariants is not exact: there are $\GG_a$-invariant 
functions $x_1^n \alpha_1$ and $x_2\alpha_2^n$ on $\mu_{\GG_a}^{-1}(0)$ which do not extend to 
$\GG_a$-invariant functions on the ambient affine space $T^*V$.
\end{ex}

There is an immediate extension of this example: we can add on copies of the trivial representation.

\begin{cor}\label{sym1 non fg}
For the $\GG_a$-representation $V = \sym^1(\CC^2) \oplus \sym^0(\CC^2)^{\oplus n}$, the ring 
of $\GG_a$-invariant functions on $H_V(0)=\mu_{\GG_a}^{-1}(0)$ is non-finitely generated. 
\end{cor}

However, this seemingly negative result is not as discouraging as one may think: by Proposition \ref{prop hypersurfaces}, 
this representation $V$ is the only case for which the 
hypersurfaces $H_V(0)$ are not normal. If we consider the induced $\GG_a$-action on the 
normalisation of $H_V(0)$ for $V = \sym^1(\CC^2) \oplus \sym^0(\CC^2)^{\oplus n}$, 
then we see that the ring of $\GG_a$-invariants is finitely generated.

\begin{lemma}\label{lemma fg norm}
Let $V = \sym^1(\CC^2) \oplus \sym^0(\CC^2)^{\oplus n}$ and $ \widetilde{H_V}(0) \ra {H_V(0)} $ be the 
normalisation of the zero level set of the moment map. 
Then $\CC[\widetilde{H_V}(0)]^{\GG_a}$ is finitely generated.
\end{lemma}
\begin{proof}
As the non-normal hypersurface $H_V(0)$ has two irreducible components, its normalisation 
is the disjoint union of the normalisation of each of these irreducible components. 
However, the irreducible components are both affine spaces of dimension $3 + 2n$ and 
so are already normal. As the $\GG_a$-action on both these affine spaces is linear, the 
ring of $\GG_a$-invariants on the normalisation is finitely generated.
\end{proof}

The next smallest example is given by $V = \sym^2(\CC^2)$, which we study in $\S$\ref{sec sym2} 
and show that the ring of $\GG_a$-invariants on $H_V(0)$ is finitely generated. If $\CC[H_V(0)]^{\GG_a}$ is finitely generated, then we can define
\[ H_V(0)/\!/\GG_a := \spec \CC[H_V(0)]^{\GG_a}, \]
which contains $H_V(0)^{\overline{s}}/\GG_a$ as an open subset. In this case, as the composition
\[ V \hookrightarrow H_V(0) \hookrightarrow T^*V \twoheadrightarrow V\]
is the identity, we have a commutative diagram 
\[\xymatrix@1{ \CC[H_V(0)]^{\GG_a}\ar[r]^{\quad \phi} &  \CC[V]^{\GG_a}  
 \\
\CC[T^*V]^{\GG_a} \ar[u] & \ar[l]_{}  \CC[V]^{\GG_a} \ar[u]_{\text{Id}}}.\]
Hence $\phi$ is surjective and corresponds to a closed immersion $V/\!/\GG_a \ra H_V(0)/\!/\GG_a$.

As we can only expect to have an algebraic symplectic structure on an open subset 
of the $\GG_a$-quotient of $H_V(0)$, the question of whether $\CC[H_V(0)]^{\GG_a}$ 
is finitely generated is not particularly pertinent. 
However, we note that the zero level set of the moment map $H_V(0)$ is the only level set 
which is preserved by the action of the upper triangular Borel $B$ in $\SL_2$. 
For representations $V$ which are not equal to a sum of trivial representations 
with one copy of $\sym^1(\CC^2)$, the hypersurface $H_V(0)$ is normal and 
irreducible.  It would be interesting to study the ring of $B$-invariants on 
$H_V(0)$ from the point of view of \cite{BK}, which in particular gives a geometrically useful finitely generated subring that should still have the desired separation properties. 

\section{Comparisons of the holomorphic symplectic quotients}

Let $\GG_a$ act linearly on an affine space $V= \AA^n$ over the complex numbers; then
\[ V/\!/\GG_a \cong W /\!/\SL_2 \]
where $W :=V \times \AA^2$. In the previous section, we studied algebraic symplectic reductions of the $\GG_a$-action on $T^*V$. We can alternatively consider algebraic symplectic reductions of the $\SL_2$-action on $T^*W$. As $\SL_2$ is semisimple, it has trivial character group and, as only $0 \in \fs \fl_2^*$ is coadjoint fixed, there is only one algebraic symplectic reduction
\[ \mu_{\SL_2}^{-1}(0)/\!/\SL_2\]
where $\mu_{\SL_2} : T^*W \ra \fs \fl_2^* $ denotes the algebraic moment map for the $\SL_2$-action on $T^*W$. Although $\GG_a$ also has trivial character group, every element in its co-Lie algebra is coadjoint fixed and we have a 1-parameter family of algebraic symplectic reductions by Theorem \ref{1-param family red}. In many ways, this reflects the fact that the representation theory of the reductive group $\SL_2$ is discrete, whereas the representation theory for the non-reductive group $\GG_a$ it is not. 

In this section, we compare the non-reductive algebraic symplectic 
quotients $\mu_{\GG_a}^{-1}(\xi)/\!/\GG_a$ with the reductive algebraic 
symplectic quotient $\mu_{\SL_2}^{-1}(0)/\!/\SL_2$. We continue to assume the $\GG_a$-representation $V$ is non-trivial.

\subsection{Equivariant morphisms between the level sets}

In order to relate the non-reductive holomorphic symplectic quotients with the 
reductive holomorphic symplectic quotient, we want to construct a $\GG_a$-equivariant 
morphism $\mu_{\GG_a}^{-1}(\xi) \ra \mu_{\SL_2}^{-1}(0)$. In this section, we assume 
that as an $\SL_2$-representation 
\[V \neq \sym^1 (\CC^2) \oplus \sym^0(\CC^2)^{\oplus n},\]
so that the hypersurface $H_V(0)$ is normal. The non-normal case will be studied 
separately in $\S$\ref{sec ex 2dim}.

We recall that $\mu_{\SL_2}^{-1}(0)$ is the closed subvariety of $T^*W$ defined by the equations
\begin{align*}
\Phi_H(x,\alpha) + u \lambda - v \eta & = 0 \\
\Phi_E(x,\alpha) + v \lambda & = 0  \\
\Phi_F(x,\alpha) + u\eta & = 0.
\end{align*}
where $\mu_{\SL_2,V}=\Phi_H H^* + \Phi_E E^* + \Phi_F F^* : T^*V \ra \fs \fl_2^*$ 
is the moment map for the $\SL_2$-action on $T^*V$ and $(x,\alpha)$ are 
coordinates on $T^*V$ and $(u,v,\lambda,\eta)$ are coordinates on $T^*\AA^2$.

\begin{rmk}
There are families of natural $\GG_a$-equivariant inclusions 
$V \hookrightarrow W = V \times \AA^2$ 
(resp. $V^* \hookrightarrow W^* = V^* \times (\AA^2)^*$ given by $x \mapsto (x,a,0)$ 
(resp. $\alpha \mapsto (\alpha,0,b)$) for constants $a,b \in \CC$, which  
give rise to inclusions $\iota_{a,b} : T^*V \hookrightarrow T^*W$. However, 
$\iota_{a,b}(\mu_{\GG_a}^{-1}(\xi)) \subset \mu_{\SL_2}^{-1}(0)$ if and only 
if $\Phi_H = \xi = 0$ and $\Phi_F = -ab$. As the functions $\Phi_H$ and $\Phi_F$ 
are not constant (cf. Example \ref{ex irred reps}), these natural families of 
$\GG_a$-equivariant maps $\iota_{a,b}$ do not map one level set to the other. 
\end{rmk} 

Fortunately, there are $\GG_a$-equivariant inclusions $\mu_{\GG_a}^{-1}(\xi) \ra \mu_{\SL_2}^{-1}(0)$ 
when $\xi = 0$. 

\begin{prop}\label{prop equiv map level sets}
Let $f : T^*V \ra T^*W$ be a morphism that extends the identity map on $T^*V$. 
Then $f$  restricts to a $\GG_a$-equivariant morphism 
$f: \mu_{\GG_a}^{-1}(\xi) \ra \mu_{\SL_2}^{-1}(0)$ only if $\xi = 0$. 
For $\xi =0$, the $\GG_a$-equivariant morphisms 
$f: \mu_{\GG_a}^{-1}(0) \ra \mu_{\SL_2}^{-1}(0)$ 
are given by two 1-parameter families of morphisms $\{i_a\}_{a \neq 0}$ and $\{j_b\}_{b \neq 0}$, 
where, for $(x,\alpha)\in \mu_{\GG_a}^{-1}(0)$,
\[ i_a (x,\alpha) := \left(x,(a,0),\alpha, 
-\tfrac{\Phi_H(x,\alpha)}{a},-\tfrac{\Phi_F(x,\alpha)}{a} \right) \]
and
\[ j_b(x,\alpha) := \left(x,
-\tfrac{\Phi_F(x,\alpha)}{b},\tfrac{\Phi_H(x,\alpha)}{b},\alpha,(0,b)\right).\]
\end{prop}
\begin{proof}
Such an extension $f$ is given by functions $f_u,f_v,f_\lambda, f_\eta\in \CC[T^*V]$. 
If $f(\mu_{\GG_a}^{-1}(\xi)) \subset \mu_{\SL_2}^{-1}(0)$, then, on $\mu_{\GG_a}^{-1}(\xi)$, we have
\begin{align*}
\Phi_H + f_u f_\lambda -f_v f_\eta =\xi + f_v f_\lambda = \Phi_F + f_u f_\eta = 0.
\end{align*}
If moreover $f :\mu_{\GG_a}^{-1}(\xi) \ra \mu_{\SL_2}^{-1}(0)$ is $\GG_a$-equivariant, then 
$f_v, f_\lambda \in \CC[\mu_{\GG_a}^{-1}(\xi)]^{\GG_a}$ and $f_u,f_\eta$ satisfy
\begin{align*}
f_u(c \cdot(x,\alpha))=& f_u(x,\alpha) + cf_v(x,\alpha)\\
f_\eta(c \cdot(x,\alpha))=& f_\eta(x,\alpha) - cf_\lambda(x,\alpha)
\end{align*}
for $c \in \GG_a$ and $(x,\alpha) \in \mu_{\GG_a}^{-1}(0)$.

First suppose that $\xi \neq 0$. Since $\mu_{\GG_a} =\xi = -f_vf_\lambda$ on 
$\mu_{\GG_a}^{-1}(\xi)$ and the equation $\mu_{\GG_a}$ is irreducible (here 
we use the assumption that $V \neq \sym^1 (\CC^2) \oplus \sym^0(\CC^2)^{\oplus n}$; 
see also Example \ref{ex irred reps}), it follows that both $f_v$ and $f_\lambda$ are 
non-zero constant functions. However, in this case, there are no functions $f_u$ 
and $f_\eta$ which transform under $\GG_a$ as above. 

Hence let $\xi =0$. As $f_vf_\lambda = 0$ on the irreducible hypersurface 
$\mu_{\GG_a}^{-1}(0)$, it follows that either $f_v = 0$ or $f_\lambda = 0$ on 
$\mu_{\GG_a}^{-1}(0)$. We will show that the first case gives rise to the family 
$i_a$; then, similarly, the second case gives rise to the family $j_b$. If $f_v = 0$, 
it follows that $f_u \in \CC[\mu_{\GG_a}^{-1}(0)]^{\GG_a}$. Since $\Phi_H$ is 
irreducible but not $\GG_a$-invariant (cf.  Example \ref{ex irred reps}) and 
$\Phi_H + f_u f_\lambda = 0$ on $\mu_{\GG_a}^{-1}(0)$, we conclude that $f_u =a$ 
for some non-zero constant $a$. It then follows that $f=i_a$.

The $\GG_a$-equivariance of $i_a$ and $j_b$ on $\mu^{-1}_{\GG_a}(0)$ can be verified 
using the following transformation properties of $\Phi_E,\Phi_F$ and $\Phi_H$:
\[ \begin{array}{l} \Phi_H(c \cdot (x,\alpha) ) =  \Phi_H(x,\alpha) + 2c \Phi_E(x,\alpha)  \\ 
\Phi_F(c \cdot (x,\alpha)) =\Phi_F(x,\alpha)- c \Phi_H(x,\alpha)-c^2 \Phi_E(x,\alpha),\end{array}\]
for $c \in \GG_a$ (which follow from the $\SL_2$-equivariance of $\mu_{\SL_2,V} : T^*V \ra \fs \fl_2^*$).
\end{proof}

Henceforth, for notational simplicity, we consider only the morphism $i:=i_1$, and we note that the same results hold, with only minor modifications to the proofs, for $i_a$ and $j_b$.

\begin{lemma}\label{stab gp pres}
The morphism $i:\mu_{\GG_a}^{-1}(0) \ra \mu_{\SL_2}^{-1}(0)$ has the following properties.
\begin{enumerate}[(i)]
\item The stabiliser groups are preserved by $i$.
\item Let $p,p' \in \mu_{\GG_a}^{-1}(0)$; then $i(p)$ and $i(p')$ lie in the same $\SL_2$-orbit 
if and only if $p$ and $p'$ lie in the same $\GG_a$-orbit.
\item For $p \in \mu_{\GG_a}^{-1}(0)$, we have $\{g \in \SL_2 : g \cdot i(p) \in i(\mu_{\GG_a}^{-1}(0))\} = \GG_a$
\end{enumerate}
\end{lemma}
\begin{proof}
For (i) suppose $p \in \mu_{\GG_a}^{-1}(0)$ has $\GG_a$-stabiliser subgroup $H$. 
If $g \in \SL_2$ fixes $i(p)$, then 
\[ g \cdot \left( \begin{array}{c} 1 \\ 0 \end{array} \right)=
 \left( \begin{array}{c} 1 \\ 0 \end{array} \right) \]
and thus $g \in \GG_a \subset \SL_2$. By equivariance of $\GG_a$, it follows that $i(p)$ has 
$\SL_2$-stabiliser group $H$.

For (ii), if $i(p)$ and $i(p')$ lie in the same $\SL_2$-orbit , then $g \cdot i(p)=i(p')$ for 
some $g \in \SL_2$. However, the same argument as above shows that $g \in \GG_a$ and so $i(g \cdot p)=i(p')$, 
but as $i$ is injective, we conclude $g \cdot p =p'$. The converse direction is trivial.

For (iii), the inclusion \lq \lq $\subseteq$'' follows by the argument used in (i), and the opposite inclusion follows by $\GG_a$-equivariance of $i$.
\end{proof}

\begin{thm}\label{thm biratl sympl}
Let $V$ be a $\GG_a$-representation and let $W = V \times \AA^2$ be the associated $\SL_2$-representation; then the non-reductive and reductive algebraic symplectic analogues are birationally symplectomorphic. More precisely, the equivariant inclusion $i: \mu_{\GG_a}^{-1}(0) \ra \mu_{\SL_2}^{-1}(0)$ induces a morphism
\[ \mu_{\GG_a}^{-1}(0)^{\overline{s}}/ \GG_a \ra \mu_{\SL_2}^{-1}(0)/\!/\SL_2\ \]
which is an isomorphism on open subsets and, moreover, this isomorphism respects the induced symplectic forms.
\end{thm}
\begin{proof}
Let $U=\mu_{\GG_a}^{-1}(0)^{\overline{s}} \cap i^{-1}(\mu_{SL_2}^{-1}(0)^{\SL_2-s})$; then $U \subset \mu_{\GG_a}^{-1}(0)$ is a $\GG_a$-invariant open subset. Let $U'= \SL_2 \cdot i(U)$; 
that is, $U'$ is the image of the morphism $U \times_{\GG_a} \SL_2 \ra \mu_{\SL_2}^{-1}(0)$ 
given by $[u,g] \mapsto g \cdot i(u)$. By construction,  
$U' \subset \mu_{\SL_2}^{-1}(0)^{\SL_2-s}$ is an $\SL_2$-invariant open subset on which $\SL_2$ acts 
with trivial stabilisers by Lemma \ref{stab gp pres} (i); therefore, $U'$ 
is smooth by Lemma \ref{alg mww} (iii). It follows from Lemma \ref{stab gp pres} (iii) that the morphism $U \times_{\GG_a} \SL_2 \ra U'$ is bijective and, moreover, is an isomorphism as $U'$ is smooth.

The $\GG_a$-action on $ U \subset \mu_{\GG_a}^{-1}(0)^{\overline{s}}$ admits a Zariski locally trivial geometric quotient and the $\SL_2$-action on $U'\subset \mu_{SL_2}^{-1}(0)^{\SL_2-s}$ admits a geometric quotient, which is a principal $\SL_2$-bundle. Moreover, we have isomorphisms
\[U/\GG_a \cong (U \times_{\GG_a} \SL_2)/\SL_2 \cong U'/\SL_2\]
which proves the desired birationality statement.  

The Liouville form $\omega_V$ on $T^*V$ (resp. $\omega_W$ on $T^*W$) induces algebraic 
symplectic forms  $\omega_{\GG_a}$ (resp. $\omega_{\SL_2}$) on $U/\GG_a$ (resp. $U'/\SL_2$) 
by Proposition \ref{nonred alg symp analog} (resp. Lemma \ref{alg mmw open}). We 
consider the commutative square
\[\xymatrix@1{ U \ar[r]^{i}\ar[d]_{\pi_{\GG_a}}& U'
\ar[d]^{\pi_{\SL_2}}  \\ U/\GG_a \ar[r]^{{\overline{i}}} & U'/\SL_2.}\]
From the explicit definition of $i : T^*V \ra T^*W$, we see that $i^*\omega_{W} = \omega_{V}$. 
Therefore
\[ \pi^*_{\GG_a}(\omega_{\GG_a})=\omega_{V}= {i}^*(\omega_{W}) = 
i^*\pi_{\SL_2}^*(\omega_{\SL_2})= \pi_{\GG_a}^* \overline{i}^*(\omega_{\SL_2}).\]
Since $ \pi_{\GG_a} : U \ra U/\GG_a$ is surjective, it follows that  
$\omega_{\GG_a} =  \overline{i}^* \omega_{\SL_2}$.
\end{proof}

For the reductive algebraic symplectic analogue, we can use the Kempf-Ness Theorem for $\SL_2$, which gives a homeomorphism between the algebraic symplectic analogue and the hyperk\"{a}hler analogue.  Then using the above birational symplectomorphism we obtain the following corollary.

\begin{cor}
The additive algebraic symplectic analogue $\mu_{\GG_a}^{-1}(0)^{\overline{s}}/ \GG_a$ of $V^{\overline{s}}/\GG_a$ admits a hyperk\"{a}hler structure on a dense open subset.  Furthermore, assuming a bounded number of trivial representations in $V$, then for all but finitely many cases, this is a codimension at least 2 complement in $\mu_{\GG_a}^{-1}(0)/\!/\ \GG_a$.
\end{cor}
\begin{proof}
Everything has been established except the codimension at least two condition.  This follows from a dimension count.  The non-stable locus for such a representation is high codimension, but $\mu_{\SL_2}^{-1}(0)$ is always codimension at most 3, and the image of $\mu_{\GG_a}^{-1}(0)$ under the inclusion is codimension at most 5, in $T^*W$.  Thus $U$ in the proof of the theorem is high codimension complement both in $\mu_{\GG_a}^{-1}(0)$ and in $\mu_{\SL_2}^{-1}(0)$ for all but finitely many cases.  By upper semi-continuity, the codimension of the complement can only increase in the quotients and analogues.  
\end{proof}

\begin{rmk}
In fact, the proof can easily be refined to handle all the cases with a computation.  It is likely that the codimension at least 2 complement result always holds.  One can refine the statement further to get explicit comparisons between the high codimension loci for the $\SL_2$ quotient analogues and the $\GG_a$ quotient analogues.  
\end{rmk}

\section{Geometric applications}

In \cite{brentnoah}, an algebraic uniformisation for an iterated blow up $X$ of a projective space along a collection of linear subspaces is provided by a description of $X$ as a non-reductive GIT quotient of an affine space $V$ by a solvable group $G=\GG_a^{t} \rtimes \GG_m^{s}$, where $s$ is the Picard rank of $X$ and $t \geq 0$. One can consider the cotangent lift of this action to $T^*V$ and take an algebraic symplectic reduction by $G$, by extending the ideas of $\S$\ref{sec G_a asr}, to construct an algebraic symplectic analogue of $X$. However, the construction of \cite{brentnoah}, often involves several copies of $\GG_a$ (that is, $t \geq 2$) or a non-linearisable action of $\GG_a$ (which is the case, for example, if $X = \overline{M}_{0,n}$ and $n \geq 6$).

The simplest geometric example covered by \cite{brentnoah}, which we can directly apply the techniques of $\S$\ref{sec G_a asr}, is $X=\text{Bl}_n\PP^{n-2}$, which is a non-reductive GIT quotient of $V=\AA^{2n}$ by a linear action of $\GG_a \rtimes (\GG_m)^{n+1}$. Let us briefly explain how this construction works. Let $X' := \text{Bl}_n\PP^{n-1}$, which is toric, as we can choose the $n$ points to be the $n$ coordinate points in $\PP^{n-1}$. The Cox ring of $X'$ is given by
\[ \Cox(X')=\CC[y_1,x_1,\dots, y_n,x_n],\]
where $y_1, \dots , y_n$ are the homogeneous coordinates on $\PP^{n-1}$ before blowing up and $x_i$ represent the exceptional divisors of each point being blown up. Then $X'$ is a GIT quotient of $\AA^{2n} =\spec\Cox(X')$ by its Picard torus $T_{X'} \cong T_X \cong \GG_m^{n+1}$, which acts as follows: for $t=(t_0,t_1, \dots , t_n) \in T_{X'}$
\[t \cdot (y_i,x_i) = (t_0\prod_{j \neq i}t_j^{-1} y_i, t_ix_i).\]
To relate $X'$ and $X$, consider the linear projection $\PP^{n-1} \dashrightarrow \PP^{n-2}$ sending the $n$ coordinate points in $\PP^{n-1}$ to the given $n$ points in $\PP^{n-2}$; this induces a rational map $X' \dashrightarrow X$, which is a Zariski-locally trivial $\AA^1$-bundle over its image. Moreover, this lifts to a $\GG_a$-action on $\AA^{2n} = \spec \Cox(X')$ given by, for $c \in \GG_a$,
\[c \cdot(y_i,x_i) = (y_i + cx_i,x_i),\]
which is normalised by the torus action. By \cite{brentnoah} Theorem 1.1, $X = \AA^n/\!/_\chi \GG_a \rtimes (\GG_m)^{n+1}$ for $\chi \in \text{Hom}(T_X,\GG_m)\cong\pic(X)$ an ample divisor class.

In this class of examples, the $\GG_a$-action on $V=\AA^{2n}$ is linear and, as an 
$\SL_2$-representation we have $V=\sym^1(\CC^2)^{\oplus n}$. In the next subsection, we 
study these representations and the algebraic symplectic analogues of $V/\!/\GG_a$ in greater 
depth.

\subsection{Representation of the form $V = \sym^1(\CC^2)^{\oplus n}$}\label{sec sym1s}

Consider the $\GG_a$-representation $V = \sym^1(\CC^2)^{\oplus n}$ and its cotangent lift to $T^*V$. More precisely, if we write 
\[ \CC[V]=\CC[y_i,x_i]_{1 \leq i \leq n} \quad \text{and} \quad \CC[T^*V]=\CC[y_i,x_i,\beta_i,\alpha_i]_{1\leq i \leq n},\]
then $c \in \GG_a$ acts by
\[c \cdot (y_i,x_i,\beta_i,\alpha_i) =(y_i+cx_i,x_i,\beta_i,\alpha_i - c \beta_i).\]
Hence, the $\GG_a$-moment map $\mu: T^*V \ra \CC$ is given by $\sum_{j=1}^n \beta_j x_j$. 

\begin{prop}\label{prop sym1}
Let $V = \sym^1(\CC^2)^{\oplus n}$ be as above. Then
\begin{enumerate}[(i)]
\item The complement to the completely stable locus for 
the $\GG_a$-action on $\mu_{\GG_a}^{-1}(0)$ is the $\GG_a$-fixed locus. In particular, the completely 
stable locus is the largest possible stable set.
\item The $\GG_a$-invariants on $T^*V$ give a separating set of invariants for the $\GG_a$-action on $\mu_{\GG_a}^{-1}(0)$ and, moreover, the nullcone of these invariants is the $\GG_a$-fixed locus.
\end{enumerate}
\end{prop}
\begin{proof}
For (i), by Lemma \ref{lemma comp stable},
\[\mu_{\GG_a}^{-1}(0) - \mu_{\GG_a}^{-1}(0)^{\overline{s}}= \{(y_i,x_i,\beta_i,\alpha_i)_{1 \leq i \leq n} \: \: : x_i = \beta_i =0 \:\:\text{for} \:\:i =1, \dots ,n \},\] 
which is precisely the $\GG_a$-fixed locus.

For (ii), we let $N({T^*V})$ (resp. $N({\mu_{\GG_a}^{-1}(0)})$) denote the nullcone, which is the common zero set of all non-constant homogeneous $\GG_a$-invariant functions on $T^*V$ (resp. $\mu_{\GG_a}^{-1}(0)$). For $T^*V$, the ring of $\GG_a$-invariant functions is finitely generated:
\[ \CC[x_i,\beta_i,x_iy_j - x_jy_i,\beta_i\alpha_j - \beta_j\alpha_i,x_i\alpha_j +y_i\beta_j] \twoheadrightarrow \CC[T^*V]^{\GG_a}\]
and $N({T^*V})=\{ (y_i,x_i,\beta_i,\alpha_i)_{1 \leq i \leq n} : x_i = \beta_i =0 \text{ for } i =1, \dots ,n \} =\mu_{\GG_a}^{-1}(0)^{\GG_a}$. As there may be additional invariants on $\mu_{\GG_a}^{-1}(0)$, which do not extend to $T^*V$, we have $N({\mu_{\GG_a}^{-1}(0)}) \subset N({T^*V})$. However, we claim in this case, that these nullcones coincide, from which Statement (ii) follows. We argue by contradiction: suppose $f \in \CC[\mu_{\GG_a}^{-1}(0)]^{\GG_a}$ and $f$ is not identically zero on $N_{T^*V}$. We can lift $f$ to a (no longer invariant) function on $T^*V$, whose derivation under the $\GG_a$-action must lie in the ideal $I:=(\sum_{i=1}^n x_i\beta_i)$. Since $f$ is not identically zero on $N_{T^*V}$, it must contain a monomial of the form $m = \prod_{i=1}^n y_i^{r_i}\alpha_i^{s_i}$. Then the derivation of $m$ is given by
\[ D(m) = \sum_{i=1}^n \prod_{j \neq i} y_j^{r_j}\alpha_j^{s_j} \left(r_ix_iy_i^{r_i -1}\alpha_i^{s_i} - s_i y_i^{r_i} \alpha_i^{s_i-1}\beta_i \right). \]
However, each monomial in $D(m)$ is not divisible by any of the monomials $x_i\beta_i$ in the equation defining $I$ and can only be \lq integrated' in one way (as only $x_i=D(y_i)$ and $\beta_i = D(-\alpha_i)$ lie in the image of $D$). Hence, no such $f$ can exist with $D(f) \in I$.
\end{proof}

This proposition justifies the following definition.

\begin{defn}
For $\GG_a$ acting on $\mu_{\GG_a}^{-1}(0)$, we define the non-reductive quotient
\[ q : \mu_{\GG_a}^{-1}(0) \ra \mu_{\GG_a}^{-1}(0)/\!/\GG_a:=\spec(\CC[T^*V]^{\GG_a}/I) \]
where $I$ is the ideal defining $\mu_{\GG_a}^{-1}(0) \subset T^*V$.
\end{defn}

\begin{rmk}
Using this definition, we have $\mu_{\GG_a}^{-1}(0)/\!/\GG_a = \mathcal{X}_0$, where $\mathcal{X}_0$ is the central fibre of the 1-parameter family of algebraic symplectic reductions given by Corollary \ref{cor central fibre}.
\end{rmk}

By Proposition \ref{prop sym1}, we have $q^{-1}(0) = \mu_{\GG_a}^{-1}(0)^{\GG_a}=\mu_{\GG_a}^{-1}(0) - \mu_{\GG_a}^{-1}(0)^s$ and 
$q(\mu_{\GG_a}^{-1}(0)^{s})=\mu_{\GG_a}^{-1}(0)^{s}/\GG_a$ has the structure of a smooth 
algebraic symplectic variety by Proposition \ref{nonred alg symp analog}. 

Let $W = V \times \AA^2$, which as an $\SL_2$-representation, we have $W =  \sym^1(\CC^2)^{\oplus n+1}$. 
Let $\mu_{\SL_2}^{-1}(0) \subset T^*W$ be the zero level set for the $\SL_2$-moment map on $T^*W$ and recall that there is a $\GG_a$-equivariant closed immersion $i : \mu_{\GG_a}^{-1}(0) \hookrightarrow \mu_{\SL_2}^{-1}(0)$ given by Proposition \ref{prop equiv map level sets}.

\begin{prop}\label{prop:2x2hyperkahler}
Let $V$ and $W$ be as above. Then
\begin{enumerate}[(i)]
\item $i(\mu_{\GG_a}^{-1}(0)^s) \subset \mu_{\SL_2}^{-1}(0)^{\SL_2-s}$,
\item $\mu_{\GG_a}^{-1}(0)^s/\GG_a$ has a hyperk\"{a}hler structure.
\end{enumerate}
\end{prop}
\begin{proof}
For (i), if $p \in \mu_{\GG_a}^{-1}(0)^s$, then there exists $f \in \CC[T^*V]^{\GG_a}$ such that 
$f(p) \neq 0$ by Proposition \ref{prop sym1}. Let $h$ be the image of $f$ under the homomorphisms
\[ \CC[T^*V]^{\GG_a} \cong \CC[T^*V \times \AA^2]^{\SL_2} \ra \CC[T^*W]^{\SL_2}, \]
where the final homomorphism is induced by an $\SL_2$-equivariant projection $\pi : T^*W \ra T^*V \times \AA^2$ such that $\pi \circ i : T^*V \ra T^*V \times \AA^2$ is given by the identity map on $T^*V$ and the point $(1,0) \in \AA^2$. Then $h(i(p))=f(p) \neq 0$. Since $T^*W = \sym^1(\CC^2)^{\oplus 2(n+1)}$, the stable set $(T^*W)^{\SL_2-s}$ is the union of all open subsets given by the non-vanishing locus of an $\SL_2$-invariant. Hence $i(p)$ is $\SL_2$-stable and has trivial stabiliser group by Lemma \ref{stab gp pres}.

For (ii), let $U=\mu_{\GG_a}^{-1}(0)^s$ and $U' = \SL_2 \cdot i(U);$ then by the argument used in Theorem \ref{thm biratl sympl}, we have
\[ U/\GG_a \cong U'/\SL_2\]
and, moreover, as $\SL_2$ acts freely on $U' \subset \mu_{\SL_2}^{-1}(0)^{\SL_2-s}$, its $\SL_2$-quotient has a hyperk\"{a}hler 
structure coming from the hyperk\"{a}hler reduction of $T^*W$ by $\SU(2)$ via the Kempf--Ness homeomorphism. 
\end{proof}

\subsection{The blow up of $n$ points on $\PP^{n-2}$}

Let us return to $\text{Bl}_n(\PP^{n-2})=V/\!/_\chi \GG_a \rtimes T$, where $V = \AA^{2n}$ and $T=\GG_m^{n+1}$. In the previous section, we described the moment map and algebraic symplectic reduction for the $\GG_a$-action on $T^*V$. In this section, we also take into account the torus action. The following lemma gives the equations for the zero level set of the algebraic moment map for the torus action.

\begin{lemma}
For $T = \GG_m^{n+1}$ acting on $T^*V$ as above, the algebraic moment map $\mu_{T} : T^*V \ra \CC^{n+1}$ 
is given by
\[ \mu_T((y_i,x_i,\beta_i,\alpha_i)_{1 \leq i \leq n})=(\sum_{i=1}^n \alpha_iy_i,\beta_1x_1 - \sum_{j \neq 1} \alpha_jy_j,....,\beta_nx_n -\sum_{j \neq n} \alpha_j y_j).\]
Hence $\mu_{T}^{-1}(0)$ is defined by the ideal
\[ I=\left(\sum_{i=1}^n \alpha_iy_i, \alpha_jy_j + \beta_j x_j :  1 \leq j \leq n \right).\]
\end{lemma}

The construction in \cite{brentnoah} as applied to this example tells us three facts we will now use.  The first is that the rational quotient map from $V$ to $\text{Bl}_n(\PP^{n-2})$ is surjective.  The second is that the action of $\GG_a \rtimes T$ is scheme-theoretically free on an open subset $U$ of $V$ that surjects onto $\text{Bl}_n(\PP^{n-2})$ under the quotient map.  The third is that this open subset $U$ is exactly the stable locus for the torus action.  

These facts allow us to describe a hyperk\"{a}hler structure on $T^*\text{Bl}_n(\PP^{n-2})$ by taking a quotient in two stages: first by the $\GG_a$-action and then by the torus action.  Observe that $\mu_{\GG_a}^{-1}(0)$ is closed under the torus action and there is an induced torus action on the quotient $\mu_{\GG_a}^{-1}(0)/\!/\GG_a$.  We know by Proposition \ref{prop sym1} and the explicit form of the fixed point locus, that the open subset $\mu_{\GG_a}^{-1}(0)^s$, which is the complement of the fixed point locus, must also be closed under the torus action.  Furthermore, the complement of $\mu_{\GG_a}^{-1}(0)^s$ in $\mu_{\GG_a}^{-1}(0)$, that is, the non-stable points for the $\GG_a$-action, are contained in the unstable points for the torus action.

\begin{lemma}
The complement $\mathcal{C}$ of the $\GG_a$-stable points in $T^*V$ is a subset of the unstable points for the torus action.  Furthermore, the image of $\mathcal{C}$ under the $\GG_a$-quotient map is contained in the unstable points for the induced torus action on $\mu_{\GG_a}^{-1}(0)/\!/\GG_a$.
\end{lemma}
\begin{proof}
By Proposition \ref{prop sym1}, the complement $\mathcal{C}$ is the nullcone of the $\GG_a$-action.  Therefore any homogeneous function on the quotient of $T^*V$ pulls-back to a $\GG_a$-invariant function that vanishes on $\mathcal{C}$.  In particular, if $f$ is a homogeneous semi-invariant function for the torus action on the quotient, then the pull-back must vanish on $\mathcal{C}$, and $f$ itself must vanish on the image of $\mathcal{C}$.  
\end{proof}

By Proposition \ref{prop:2x2hyperkahler} there is a hyperk\"{a}hler structure on $\mu_{\GG_a}^{-1}(0)^s/\GG_a$.  By the above lemma, the GIT torus quotient of $\mu_{\GG_a}^{-1}(0) /\!/ \GG_a$ is equal to the GIT torus quotient of $\mu_{\GG_a}^{-1}(0)^s / \GG_a$.  By direct observation, the solvable group acts freely on $\mu_{\GG_a}^{-1}(0)^s$ and hence the torus acts freely on $\mu_{\GG_a}^{-1}(0)^s / \GG_a$.  Therefore one can now apply the standard torus hyperk\"{a}hler quotient construction to $\mu_{\GG_a}^{-1}(0) /\!/\ \GG_a$.  By our above facts, this gives us a hyperk\"{a}hler structure on $T^*\text{Bl}_n(\PP^{n-2})$ by restriction to an open set: first, Proposition \ref{nonred alg symp analog} tells us $T^*(V^s/\GG_a)$ is an open subset of $\mu_{\GG_a}^{-1}(0)^s / \GG_a$; then, the further torus quotient analogue must contain $T^*(V^s/(\GG_a \rtimes T)$, which equals $T^*\text{Bl}_n(\PP^{n-2})$, as an open subset.  

\begin{prop}
The cotangent bundle $T^* \text{Bl}_n(\PP^{n-2})$ has a hyperk\"{a}hler structure realised as an open subset of an algebraic symplectic analogue of a quotient of affine space by a solvable group action. 
\end{prop}


\section{Examples}\label{sec exs}

We can explicitly study some examples, using an algorithm of van den Essen \cite{vandenEssen} and Derksen \cite{derksen} to compute rings of $\GG_a$-invariants on the zero level set of the moment map. Let us quickly explain the algorithm.

For a $\GG_a$-action on an affine variety, the algorithm produces a list of generators for $\CC[X]^{\GG_a}$, which terminates if and only if the ring is finitely generated. The first step is to choose $f \in \CC[X]^{\GG_a}$ 
such that on $X_f$ the $\GG_a$-action becomes a trivial $\GG_a$-torsor (the existence of such an $f$ is 
a consequence of the derivation of the $\GG_a$-action being locally nilpotent). Then, as taking $\GG_a$-invariants is left exact, there is an inclusion 
\[ \CC[X]^{\GG_a} \hookrightarrow \CC[X]_f^{\GG_a}\] 
and $\CC[X]^{\GG_a}=\CC[X] \cap \CC[X]_f^{\GG_a}$. 
The algorithm gives generators for this intersection using colon ideals and can be performed using a computer algebra package, which is adept at computing Groebner bases; for details see \cite{derksen} Algorithms 2.6 and 2.7.

\subsection{The 2 dimensional irreducible $\SL_2$-representation}\label{sec ex 2dim}

Let $V = \sym^1(\CC^2) \cong \AA^2$ with 
$\GG_a$-action given by $c \cdot (x_1,x_2) = (x_1+cx_2,x_2)$, for $c \in \GG_a$. 
We have GIT quotient $\pi_V : V \ra V/\!/\GG_a =\spec \CC[x_2] \cong \AA^1$, which 
restricts to a geometric quotient on $V^s = \{(x_1,x_2): x_2 \neq 0 \}$. 

The moment map for the cotangent lift of the $\SL_2$-action to $T^*V$ is given by
\[ \mu_{\SL_2,V}(x,\alpha) = (x_1\alpha_1 -x_2 \alpha_2)H^* + (x_2\alpha_1)E^* + (x_1\alpha_2)F^*\]
and the $\GG_a$-moment map is given by $\mu_{\GG_a}(x,\alpha) = x_2\alpha_1$. 
By Example \ref{ex non fg}, $\CC[\mu_{\GG_a}^{-1}(0)]^{\GG_a}$ is non-finitely generated and so we instead work with the normalisation 
\[ \widetilde{\mu_{\GG_a}^{-1}(0)} = \mu_{\GG_a}^{-1}(0)_{(1)} \bigsqcup \mu_{\GG_a}^{-1}(0)_{(2)}\]
where $\mu_{\GG_a}^{-1}(0)_{(1)} = \spec \CC[x_1,\alpha_1,\alpha_2]$ has GIT quotient 
$\mu_{\GG_a}^{-1}(0)_{(1)}/\!/\GG_a = \spec \CC[x_1,\alpha_1]$ and 
$\mu_{\GG_a}^{-1}(0)_{(2)} = \spec \CC[x_1,x_2,\alpha_2]$ has GIT quotient 
$\mu_{\GG_a}^{-1}(0)_{(2)}/\!/\GG_a = \spec \CC[x_2,,\alpha_2]$, and
\[ \mu_{\GG_a}^{-1}(0)_{(1)}^{\overline{s}} = \{ (x_1,\alpha_1,\alpha_2) : \alpha_1 \neq 0 \} \quad \text{and} \quad 
 \mu_{\GG_a}^{-1}(0)_{(1)}^{\overline{s}} = \{ (x_1,x_2,\alpha_2) : x_2 \neq 0 \}.\]

Let $W = V \times \AA^2$ and write $\CC[T^*W] =\CC[T^*V][u,v,\lambda,\eta]$. Then for the $\SL_2$-action 
on $T^*W$, the zero level set $\mu_{\SL_2}^{-1}(0) \subset T^*W$ is defined by the ideal
\[ (  x_1\alpha_1-x_2\alpha_2 +\lambda u-\eta v ,x_2\alpha_1 +\lambda v,x_1\alpha_2 +\eta u)\]
and $\CC[T^*W]^{\SL_2}=\CC[h_1,\cdots,h_6]/(h_1h_4 - h_2h_5 +h_3h_6)$ where
\[\begin{array}{lll} 
h_1 = x_2 u -x_1v \quad & 
h_2 = \alpha_1 u + \alpha_2 v \quad & 
h_3 = x_1\alpha_1 + x_2\alpha_2 \\
h_4 = \alpha_1 \eta - \alpha_2 \lambda \quad &
h_5 = x_1 \lambda +x_2 \eta \quad & 
h_6 = u \lambda + v \eta. 
\end{array}\]
We have $\mu_{\SL_2}^{-1}(0)^{\SL_2-s}= \cup_{i=1}^6 \mu_{\SL_2}^{-1}(0)_{h_i}$.

The $\GG_a$-equivariant closed embedding $i : \mu_{\GG_a}^{-1}(0) \ra \mu_{\SL_2}^{-1}(0)$ induces  \[\overline{i_{(j)}}: \mu_{\GG_a}^{-1}(0)_{(j)}/\!/\GG_a \ra \mu_{\SL_2}^{-1}(0)/\!/\SL_2\] 
where
\begin{align*}
 \overline{i_{(1)}}(x_1,\alpha_1) = &(0, \alpha_1,\alpha_1x_1,0, -\alpha_1x_1^2,-\alpha_1x_1),\\
 \overline{i_{(2)}}(x_2,\alpha_2) = &(x_2,0,x_2\alpha_2,-x_2\alpha_2^2,0,x_2\alpha_2). 
\end{align*}
We have $i_{(j)}(\mu_{\GG_a}^{-1}(0)_{(j)}) \subset \mu_{\SL_2}^{-1}(0)^{\SL_2-s}_I$ and the 
morphisms $\overline{i_{(j)}}$ induce birational symplectomorphisms between the reductive and non-reductive algebraic symplectic analogue. Furthermore, for $V/\!/\GG_a \cong \AA^1 \cong W/\!/\GG_a$, we see that the non-reductive algebraic symplectic 
analogue $\mu_{\GG_a}^{-1}(0)_{(2)}/\!/\GG_a \cong \AA^2$ provides a better algebraic symplectic analogue 
that the reductive algebraic symplectic analogue $ \mu_{\SL_2}^{-1}(0)/\!/\SL_2$, which is a singular subvariety of $\AA^6$.

\subsection{The 3 dimensional irreducible $\SL_2$-representation}\label{sec sym2}

Let $V = \sym^2(\CC^2) \cong \AA^3$ with
$\GG_a$-action given by $c \cdot (x_1,x_2,x_3) = (x_1+2cx_2+c^2x_3,x_2+cx_3,x_3)$, for $c \in \GG_a$, and $\SL_2$-action
\[ \left(\begin{array}{cc} p & q \\ r & s \end{array} \right) \mapsto \left( \begin{array}{ccc} p^2 & 2pq & q^2 \\ pr & ps+qr & qs \\ r^2 & 2rs & s^2 \end{array}\right).\]

The non-reductive quotient $V \ra V/\!/\GG_a= \spec \CC[x_3, x_1x_3 -x_2^2] \cong \AA^2$ 
restricts to a geometric quotient on $V^{s}=\{(x_1,x_2,x_3) : x_3 \neq 0 \}$.

The algebraic moment map for $\SL_2$ acting on $T^*V$ is given by the holomorphic 
functions 
\[ \Phi_H(x,\alpha) =2x_1\alpha_1-2x_3\alpha_3 \quad 
\Phi_E(x,\alpha)=\mu_{\GG_a}(x,\alpha) =2x_2\alpha_1 +\alpha_2x_3
\quad \Phi_F(x,\alpha) = x_1\alpha_2 +2 \alpha_3 x_2.\]
For the $\GG_a$-action on $\mu_{\GG_a}^{-1}(0)$, we have
\[ \mu_{\GG_a}^{-1}(0) -\mu_{\GG_a}^{-1}(0)^{\overline{s}} = \{ (x,a) \in \mu_{\GG_a}^{-1}(0) : x_3 = a_1 = 0 \} \]
which strictly contains 
\[ \mu_{\GG_a}^{-1}(0)^{\GG_a} = \sing \mu_{\GG_a}^{-1}(0) = \{ (x,a) \in \mu_{\GG_a}^{-1}(0): x_2 = x_3 = a_1 = a_2 =0\}.\]
By the algorithm of van den Essen and Derksen, $\CC[\mu_{\GG_a}^{-1}(0)]^{\GG_a}$ is finitely generated by
\[ \begin{array}{lll}
f_1= x_3, & \quad f_4 = x_1a_1 +x_2a_2+x_3a_3, & \quad 
f_7 =4x_3a_3^2 + 2x_2a_2a_3 -4x_1 a_1 a_3 +x_1a_2^2. \\
f_2 = x_1x_3 -x_2^2, & \quad f_5 = 4a_1a_3 -a_2^2, & \quad 
f_8 = 2x_3x_1a_3 - 2x_2^2a_3 -2x_1^2a_1 -x_1x_2a_2. \\
f_3 =  a_1, & \quad f_6 = x_1a_1 - x_3a_3. & 
\end{array} \]
The invariants $f_6,f_7$ and $f_8$ are intrinsic to the hypersurface, but one easily checks that these do not give any additional separation or increase the stable locus; i.e., the locally trivial stable set coincides with the completely stable set in this example.

For the $\SL_2$-action on $T^*W$, the zero level set $\mu_{\SL_2}^{-1}(0) \subset T^*W$ is defined by 
the ideal
\[( 2x_1\alpha_1-2x_3\alpha_3 +\lambda u-\eta v , 2x_2\alpha_1+ \alpha_2 x_3 +\lambda v , x_1\alpha_2 + 2 \alpha_3x_2 +\eta u)\]
and $\CC[\mu_{\SL_2}^{-1}(0)]^{\SL_2}$ is generated by
\[\begin{array}{ll} 
h_1 = x_1 v^2 - 2x_2uv +x_3 u^2 \quad & 
h_2 = x_1x_3 - x_2^2  \\
h_3 = \alpha_1 u^2 + \alpha_2 uv + \alpha_3 v^2 \quad &
h_4 = x_1\alpha_1 + x_2 \alpha_2 + x_3 \alpha_3 \\
h_5 = 4 \alpha_1 \alpha_3 - \alpha_2^2 \quad & 
h_6 = u \lambda + v \eta \\
h_7 = \alpha_1 \eta^2 -\alpha_2 \eta \lambda + \alpha_3 \lambda^2 \quad &
h_8 = x_1\lambda^2  + 2x_2 \eta \lambda + x_3 \eta^2 \\
h_9 = x_1 v \lambda -x_2(u\lambda -v \eta) -x_3 u \eta \quad &
h_{10} = 2\alpha_1 u \eta -\alpha_2 (u\lambda -v \eta) - 2\alpha_3 \lambda v.
\end{array} \]

The birational symplectomorphism $\overline{i} : \mu_{\GG_a}^{-1}(0)/\!/\GG_a \ra \mu_{\SL_2}^{-1}(0)/\!/\SL_2$ given by Theorem \ref{thm biratl sympl} can be explicitly determined:
\[  \overline{i} (y_1, \dots, y_8)=(y_1,y_2,y_3,y_4,y_5, -2y_6, -y_6y_7, -2y_6y_8,0,0).\]

\bibliographystyle{amsplain}
\bibliography{references}

\medskip \medskip

\noindent{Brent Doran (\texttt{brent.doran@math.ethz.ch}) ETH Zurich, Switzerland.} 

\medskip

\noindent{Victoria Hoskins (\texttt{hoskins@math.fu-berlin.de}) Freie Universit\"{a}t Berlin, Germany.}

\end{document}